\documentclass[11pt,a4paper,twoside]{article}
\usepackage{latexsym,amsfonts,amsmath, amsthm, amssymb}
\usepackage{fullpage}
\usepackage{xcolor,bm, bbm}
\usepackage[utf8x]{inputenc}
\usepackage{array}
\usepackage{mathrsfs}

\usepackage{fourier}

\newcommand{\R}{\mathbb R}
\newcommand{\N}{\mathbb N}

\newcommand{\E}{\mathbb E}
\newcommand{\Pro}{\mathbb P}

\newcommand{\Var}{\mathrm{Var}}

\newcommand{\vol}{\mathrm{vol}}

\def\dint{\textup{d}}
\newcommand{\SSS}{\ensuremath{{\mathbb S}}}

\newcommand{\Cov}{\operatorname{Cov}}

\newtheorem{thm}{Theorem}[section]

\newtheorem{lemma}[thm]{Lemma}
\newtheorem{df}[thm]{Definition}

\newtheorem{thmalpha}{Theorem}

{
\theoremstyle{definition}

\newtheorem{rmk}[thm]{Remark}

}
\def\bC{\mathbf{C}}

\allowdisplaybreaks
\begin{document}

\title{Thin-shell concentration for random vectors in Orlicz balls via moderate deviations and Gibbs measures}

\medskip

\author{David Alonso-Guti\'errez and Joscha Prochno}



\date{}

\maketitle

\begin{abstract}
\small
In this paper, we study the asymptotic thin-shell width concentration for random vectors uniformly distributed in Orlicz balls. We provide both asymptotic upper and lower bounds on the probability of such a random vector $X_n$ being in a thin shell of radius $\sqrt{n}$ times the asymptotic value of $n^{-1/2}\left(\E\left[\Vert X_n\Vert_2^2\right]\right)^{1/2}$ (as $n\to\infty$), showing that in certain ranges our estimates are optimal. In particular, our estimates significantly improve upon the currently best known general Lee-Vempala bound when the deviation parameter $t=t_n$ goes down to zero as the dimension $n$ of the ambient space increases. We shall also determine in this work the precise asymptotic value of the isotropic constant for Orlicz balls. Our approach is based on moderate deviation principles and a connection between the uniform distribution on Orlicz balls and Gibbs measures at certain critical inverse temperatures with potentials given by Orlicz functions, an idea recently presented by Kabluchko and Prochno in [The maximum entropy principle and volumetric properties of Orlicz balls, J. Math. Anal. Appl. {\bf 495}(1) 2021, 1--19].
\medspace
\end{abstract}


\section{Introduction and main results}

Already the early years in the local theory of Banach spaces and geometric functional analysis have demonstrated a deep connection between the geometry of finite-dimensional normed spaces and probability theory. Powerful methods have been developed on both sides and new fields at the crossroads of functional analysis, discrete and convex geometry, and probability theory have emerged. Two such fields are asymptotic geometric analysis and high-dimensional probability theory, which overlap in many different ways while still leaning towards the directions the names already indicate. Driving forces behind many research activities in both fields have their origin in applied sciences, for instance, in form of the Kannan-Lov\'asz-Simonovits (KLS) conjecture in theoretical computer science \cite{AGB2015, LV2019_survey}, and applications in other areas of mathematics are manifold (see, e.g., \cite{FR2013, GV2016, HKNPU2019, HPU2019, M2015, V2018} and the references cited therein.).

The past decades have shown the fundamental importance of central limit phenomena for both fields, the most prominent example being arguably the central limit theorem for convex bodies due to Klartag \cite{K2007}, which says that most marginals of an isotropic convex body in high dimensions are close to a Gaussian distribution. Beyond that, various geometric quantities have been shown to follow a central limit theorem as the dimension of the ambient space tends to infinity, e.g., \cite{AGetal2019, APT2019, BKPST2020, GKT2019, JP2019, KLZ15pomi, KPT2019_I, KPT2019_II, PPZ14, R2005, Schmu2001,Stam82, ThaeleTurchiWespi}, and aside from the universality they describe, which no doubt is a beautiful and fascinating property in its own right, those weak limit theorems find applications in different situations, e.g., \cite{AP2015, KPT2019_I, Schmu2001}. What many of those results have in common and what makes their proofs more delicate is that the source of the Gaussian approximation is not attributed to independence, or a weak form of independence, but rather to geometry and more specifically convexity. For instance, pivotal to Klartag's proof of the central limit theorem is the following principle going back to Sudakov \cite{S1978}, Diaconis and Freedman \cite{DF1984}, and von Weizs\"acker \cite{vW1997}, which had been put forward again by Anttila, Ball, and Perissinaki \cite{ABP2003}: an isotropic  random vector $X\in\R^n$, i.e., a centered random vector $X\in\R^n$ with identity covariance matrix, has most marginals approximately Gaussian if $\|X\|_2/\sqrt{n}$ concentrates around $1$, i.e., $\|X\|_2$ concentrates in a thin shell of radius $\sqrt{n}$ and `small' width (see \cite[Theorem 1.4]{K2007} and \cite[Chapter 12]{IsotropicConvexBodies}). This principle has led to the {\it thin-shell width conjecture}, which proposes the existence of an absolute constant $C\in(0,\infty)$ such that for every $n\in\N$ and every isotropic random vector $X\in\R^n$ one has $\E[\Vert X\Vert_2-\sqrt{n}]^2\leq C$.
This conjecture is known to be equivalent (see, for instance, \cite{AGB2013}) to the so-called {\it variance conjecture}, which was formally conjectured in \cite{BK2003} and proposes the existence of an absolute constant $C\in(0,\infty)$ such that for every $n\in\N$ and every isotropic random vector $X\in\R^n$ one has $\Var[\Vert X\Vert_2^2]\leq Cn$.
It has been verified for random vectors uniformly distributed on unconditional bodies \cite{K2009} (see also \cite{ABP2003}, \cite{LW2008}, \cite{S2008}, and \cite{W2007} for previous results on random vectors uniformly distributed on the $\ell_p^n$ balls) as well as for generalized Orlicz balls \cite{MK2018} and random vectors uniformly distributed on the regular simplex \cite{BW2009}. The best general estimate known up to now is due to Lee-Vempala \cite{LV2018}, who proved that any $n$-dimensional isotropic random vector verifies that $\E[\Vert X\Vert_2-\sqrt{n}]^2\leq C\sqrt{n}$, where $C\in(0,\infty)$ is an absolute constant, improving the previous estimate $\E[\Vert X\Vert_2-\sqrt{n}]^2\leq Cn^{2/3}$ given by Gu\'edon-Milman \cite{GM2011}. In the same paper, Lee and Vempala showed that for any isotropic random vector in $\R^n$ and any $t\in(0,\infty)$,
\begin{equation}\label{eq:Concentration Lee-Vempala}
\Pro\left[\left|\frac{\Vert X\Vert_2}{\E[\Vert X\Vert_2]}-1\right|\geq t\right]\leq e^{-c\min\{t,t^2\}\sqrt{n}},
\end{equation}
where $c\in(0,\infty)$ is an absolute constant, thereby improving for small values of $t$ the value of the exponent in $t$ in such a concentration inequality from $t^3$, proven by Gu\'edon and Milman, to $t^2$. More precisely, Gu\'edon and Milman proved that for any isotropic random vector in $\R^n$ and any $t\in(0,\infty)$,
\begin{equation}\label{eq:Concentration Guedon-Milman}
\Pro\left[\left|\frac{\Vert X\Vert_2}{\sqrt{n}}-1\right|\geq t\right]\leq e^{-c\min\{t,t^3\}\sqrt{n}},
\end{equation}
where $c\in(0,\infty)$ is an absolute constant. In the case of the $\ell_p^n$ balls sharper concentration results have been obtained for the $\ell_q$-norm of a random vector uniformly distributed on the $\ell_p^n$ sphere \cite{N2007} and, as explained in \cite{SZ2000}, the estimates translate immediately to random vectors uniformly distributed on  $B_p^n$, the unit ball of $\ell_p^n$. For more details we refer to, e.g., \cite[Section 2]{AGB2018}. In particular, if $p\geq 2$ and $X$ is a random vector uniformly distributed on $B_p^n$, then
\begin{equation}\label{eq:Concentration Schechtamnn-Zinn-Naor-pgeq2}
\Pro\left[\left|\frac{\Vert X\Vert_2^2}{\E[\Vert X\Vert_2^2]}-1\right|\geq t\right]\leq 12e^{-c\min\left\{t,t^2\right\}n},
\end{equation}
where $c\in(0,\infty)$ is an absolute constant and if $1\leq p<2$, then
\begin{equation}\label{eq:Concentration Schechtamnn-Zinn-Naor-leq2}
\Pro\left[\left|\frac{\Vert X\Vert_2^2}{\E[\Vert X\Vert_2^2]}-1\right|\geq t\right]\leq Ce^{-c\psi(n,t)},
\end{equation}
where $c,C\in(0,\infty)$ are absolute constants and $\psi(n,t)$ is a function that takes different forms in different ranges of $t$ with respect to $n$.


Other types of limit theorems, namely moderate and large deviation principles, describing the fluctuations beyond the Gaussian scale, have only recently been obtained for quantities studied in asymptotic geometric analysis and high-dimensional probability theory after their introduction by Gantert, Kim, and Ramanan \cite{GKR2017} and Kabluchko, Prochno, and Th\"ale \cite{KPT2019_II}. Contrary to the universality in a central limit theorem, which comes at the price of information regarding the underlying distribution being lost in the limit, moderate and large deviations are sensitive and typically parametric in view of the underlying random objects, meaning that in a our context they still encode subtle geometric information. While the full strength of this fact regarding applications in asymptotic geometric analysis and high-dimensional probability is yet to figure out, Alonso-Guti\'errez, Prochno, and Th\"ale have recently discovered in \cite{AGPT2021}, using a theorem of Gromov and Milman \cite{GM1987}, a connection between the study of moderate and large deviations for isotropic log-concave random vectors and the famous KLS conjecture (which is stronger than both the thin-shell width and variance conjecture). While most of the initial works on moderate and large deviations in the geometric framework had been restricted to $\ell_p^n$ balls (see, e.g., \cite{APT2018,GKR2016,GKR2017,KPT2019_cube,KPT2019_I,KPT2019_II,KR2018,LR2020}), because of a useful probabilistic representation of Schechtman and Zinn \cite{SchechtmanZinn} which allowed for certain explicit computations to be carried out, this had been overcome by Kabluchko, Prochno, and Th\"ale in their work on Sanov-type large deviations for the Schatten classes \cite{KPT2019_sanov}, and was recently pushed further by Kim, Liao, and Ramanan \cite{KLR2019}. In the updated version of their paper on arXiv they obtained large deviation principles in the general setting of Orlicz balls by using a method similar to the one that has recently been put forward by Kabluchko and Prochno in \cite{KP2020} who studied the asymptotic volumetric properties of Orlicz balls. The approach, as is explained in \cite[Section 1.2]{KP2020}, is based on a connection between random vectors in unit balls of Orlicz spaces and certain Gibbs measures whose potentials are given by the respective Orlicz functions and rests on the maximum entropy principle from large deviations theory and statistical mechanics. In particular, the connection also explains why the probabilistic representation of Schechtman and Zinn is so intimately related to the geometry of $\ell_p^n$ balls. In this paper, we use this connection between the uniform distribution on Orlicz balls and Gibbs measures with Orlicz potentials from \cite{KP2020} to study the concentration of random vectors in Orlicz balls in a thin-shell with radius $\sqrt{n}$ times the asymptotic value of $n^{-1/2}\left(\E\left[\Vert X_n\Vert_2^2\right]\right)^{1/2}$ and obtain in several cases strong and even sharp asymptotic estimates, extending the concentration results in \eqref{eq:Concentration Schechtamnn-Zinn-Naor-pgeq2} from $\ell_p^n$ balls to Orlicz balls when $t=t_n$ goes down to zero. In this situation we exploit the normalization of a random variable on the scale of moderate deviations to obtain our estimates.  The approach also enables us to compute the precise asymptotic value of the isotropic constant for Orlicz balls. Other than the explained connection to Gibbs measures, our proofs are based on moderate and sharp large deviation techniques for sums of independent random variables due to Petrov \cite{P1965, P1975} and Eichelsbacher and L\"owe \cite{EL2003}.

\subsection{The main results}

In order to state the main results of this paper, we first need to introduce some notions and notation. A function $M:\R\to\R$ is called an Orlicz function if $M(0)=0$, $M(t)>0$ for $t\neq 0$, and $M$ is even and convex. For $R\in(0,\infty)$, let us denote by $B_M^n(nR)$ the Orlicz ball
\[
B_M^n(nR) := \Bigg\{x=(x_i)_{i=1}^n \in\R^n\,:\, \sum_{i=1}^n M(x_i) \leq nR \Bigg\}.
\]
Note that if $M(x)=|x|^p$ for $1\leq p<\infty$, then we obtain an $\ell_p^n$ ball of radius $(nR)^{1/p}$.
The isotropic constant of the Orlicz ball $B_M^n(nR)$ is the number $L_{B_M^n(nR)}\in(0,\infty)$ such that
\[
nL_{B_M^n(nR)}^2=\frac{1}{\vol_n(B_M^n(nR))^{1+2/n}}\int_{B_M^n(nR)} \Vert x\Vert_2^2 dx.
\]

For more background on the isotropic constant, we refer the reader to Subsection \ref{subsec:isotropic convex bodies} below and the references provided there.

In what follows, given an Orlicz function $M:\R\to\R$ and a radius $R\in(0,\infty)$, we shall denote by $\varphi_M$ the log-partition function with potential $M$, i.e.,
$$
\varphi_M:(-\infty,0)\to\R,\qquad \varphi_M(\alpha)=\log \int_\R e^{\alpha M(x)}dx,
$$
and by $p_M:\R\to[0,\infty)$ the log-concave Gibbs density with potential $M$, i.e.,
$$
p_M(x)=e^{\alpha_*M(x)-\varphi_M(\alpha_*)}=\frac{ e^{\alpha_* M(x)}}{\int_\R e^{\alpha_* M(x)}dx},\quad x\in\R,
$$
where $\alpha_*\in(-\infty,0)$ is the unique element, in statistical mechanics parlance the (critical) inverse temperature, at which
\[
\varphi_M^\prime(\alpha_*)=\frac{\int_\R M(x)e^{\alpha_*M(x)}dx}{\int_\R e^{\alpha_*M(x)}dx}=R
\]
holds.

Our main results are the following. The first determines the precise asymptotic value of the isotropic constant for Orlicz balls $B_M^n(nR)$.

\begin{thmalpha}\label{thm:isotropic constant orlicz}
Let $M:\R\to[0,\infty)$ be an Orlicz function and $R\in(0,\infty)$.
Then
$$
\lim_{n\to\infty}L_{B_M^n(nR)}=\frac{e^{\alpha_*R}}{\int_\R e^{\alpha_*M(x)}dx}\left(\frac{\int_\R x^2e^{\alpha_* M(x)}dx}{\int_\R e^{\alpha_* M(x)}dx}\right)^{1/2}.
$$
\end{thmalpha}

This result will be obtained as a consequence of a concentration result for a random vector uniformly distributed on $B_M^n(nR)$ on a thin-shell of radius $\sqrt{n}L_Z$. The following result provides a much sharper concentration estimate than the general situation under an extra assumption on the growth of the function $M$.

\begin{thmalpha}\label{thm:concentration thin-shell orlicz}
Let $M:\R\to[0,\infty)$ be an Orlicz function such that $M\in\Omega(x^2)$ as $x\to\infty$, $n\in\N$, $R\in(0,\infty)$. Assume that $Z$ is a random variable with Gibbs density $p_M$ and $X_n$ a random vector uniformly distributed on $B_M^n(nR)$. Then, for every sequence $(t_n)_{n\in\N}\in(0,\infty)^\N$ such that $\frac{1}{\sqrt{n}}\ll t_n\ll1$, we have that, as $n\to\infty$,
$$
\Pro\left[\Bigg|\frac{\Vert X_n\Vert_2^2}{nL_Z^2}-1\Bigg|\geq t_n\right]\leq|\alpha_*|\sqrt{2\pi n\,\varphi_M^{\prime\prime}(\alpha_*)}e^{-\frac{t_n^2nL_Z^4(1+o(1))}{2\Var[Z^2]}}\big(1+o(1)\big),
$$

where
$$
L_Z^2:=\frac{\int_\R x^2e^{\alpha_* M(x)}dx}{\int_\R e^{\alpha_* M(x)}dx}.
$$
\end{thmalpha}

In the proof of these results, we use ideas recently put forward by Kabluchko and Prochno in \cite{KP2020}. As explained in Section 1.2 of their paper, a probabilistic approach to the geometry of Orlicz balls (and in particular to their asymptotic volumetric properties) is naturally associated to Gibbs distributions at certain critical inverse temperatures with potentials given by Orlicz functions. Underlying this connection is the maximum entropy principle from statistical mechanics (see \cite{RAS2015}). In the very same spirit, we use those Gibbs distributions to determine the asymptotic value of the isotropic constant for Orlicz balls, which, to the reader familiar with those distributions, is apparent by merely looking at the quantities appearing in Theorem \ref{thm:isotropic constant orlicz}. Another ingredient in the proofs is the use of moderate deviation principles, with which we can get the optimal constants in some inequalities. An alternative use of Bernstein's inequality would provide a similar estimate with slightly worse constants.


Having presented Theorem \ref{thm:concentration thin-shell orlicz}, a natural question that arises now is how good the bound presented there really is. As it turns out, whenever the sequence $t_n$ is not getting too small, i.e., $\frac{1}{n^{1/4}}\ll t_n \ll 1$, on the exponential scale we can prove a matching lower bound on the upper tail concentration probability. This is part of the following result, which establishes something slightly more general. We shall briefly discuss the conditions relating the sequences $(r_n)_{n\in\N}$ and $(t_n)_{n\in\N}$ appearing in the statement in Remark \ref{rem:discussion of conditions} below.

\begin{thmalpha}\label{thm:lower bound}
Let $M:\R\to[0,\infty)$ be an Orlicz function such that $M\in\Omega(x^2)$ as $x\to\infty$, $n\in\N$, $R\in(0,\infty)$. Assume that $Z$ is a random variable with Gibbs density $p_M$ and $X_n$ a random vector uniformly distributed on $B_M^n(nR)$. Then, for every two sequences $(r_n)_{n\in\N},(t_n)_{n\in\N}\in(0,\infty)^\N$ such that $\frac{1}{\sqrt{n}}\ll\frac{r_n}{n}\ll t_n\ll1$, we obtain, as $n\to\infty$,
$$
\Pro\left[\Bigg|\frac{\Vert X_n\Vert_2^2}{nL_Z^2}-1\Bigg|\geq t_n\right]\geq |\alpha_*|\sqrt{2\pi n\varphi^{\prime\prime}(\alpha_*)}e^{-r_n\big(-\alpha_*+o(1)\big)}\big(1+o(1)\big),
$$
where
$$
L_Z^2:=\frac{\int_\R x^2e^{\alpha_* M(x)}dx}{\int_\R e^{\alpha_* M(x)}dx}.
$$
In particular, whenever $\frac{1}{n^{1/4}}\ll t_n\ll1$ (and choosing $r_n=t_n^2n$), then we obtain
$$
\Pro\left[\Bigg|\frac{\Vert X_n\Vert_2^2}{nL_Z^2}-1\Bigg|\geq t_n\right]\geq |\alpha_*|\sqrt{2\pi n\varphi^{\prime\prime}(\alpha_*)}e^{-t_n^2nL_Z^4\big(-\alpha_*+o(1)\big)}\big(1+o(1)\big).
$$
\end{thmalpha}

Last but not least, we study the thin-shell concentration in the case of $\ell_p^n$ balls. Obviously, $\ell_p^n$ balls are Orlicz balls for the function $M(t)=|t|^p$, $1\leq p< \infty$ and this means that Theorem \ref{thm:concentration thin-shell orlicz} carries over to the case $p\geq2$, providing in our setting concentration results similar to \eqref{eq:Concentration Schechtamnn-Zinn-Naor-pgeq2} on a thin shell of a slightly different radius. However, they do not give anything when $p<2$. Still, using a result on moderate deviations for independent and identically distributed random variables due to Eichelsbacher and L\"owe \cite{EL2003}, we can get an asymptotic upper bound on the thin-shell width concentration for a restricted $t_n$-range, which still allows us to get as close to $\frac{1}{\sqrt{n}}$ as we wish. In fact, we provide an asymptotically matching lower bound for $p$ and $t_n$ in a certain range. Our result reads as follows.

\begin{thmalpha}\label{thm:thin-shell concetration lp balls 1<p<2}
Let $1\leq p <2$, $n\in\N$, and $X_n$ be a random vector uniformly distributed on $B_p^n(n):=B_{|\cdot|^p}^n(n)$. Assume that $(t_n)_{n\in\N}\in(0,\infty)^\N$ is a sequence such that $\frac{1}{\sqrt{n}}\ll t_n\ll\frac{n^\frac{p}{2(4-p)}}{\sqrt{n}}=\frac{1}{n^{\frac{4-2p}{2(4-p)}}}$. Then, as $n\to\infty$,
$$
\Pro\left[\Bigg|\frac{\Vert X_n\Vert_2^2}{nL_Z^2}-1\Bigg|\geq t_n\right]\leq\sqrt{\frac{2\pi n}{p}}e^{-\frac{t_n^2nL_Z^4(1+o(1))}{2\Var(Z^2)}}\big(1+o(1)\big),
$$

where $Z$ is a $p$-generalized Gaussian random variable with density $p(x)=\frac{e^{-\frac{|t|^p}{p}}}{2p^{1/p}\Gamma\left(1+\frac{1}{p}\right)}dt$ and then
$$
L_Z^2:=\frac{p^{2/p}\Gamma\left(1+\frac{3}{p}\right)}{3\Gamma\left(1+\frac{1}{p}\right)}\quad\textrm{and}\quad\Var[Z^2]=\frac{p^{4/p}\left(9\Gamma\left(1+\frac{5}{p}\right)\Gamma\left(1+\frac{1}{p}\right)-5\Gamma\left(1+\frac{3}{p}\right)^2\right)}{45\Gamma\left(1+\frac{1}{p}\right)^2}.
$$
Furthermore, if $\,\frac{4}{3}<p<2$ and $\frac{1}{n^{1/4}}\ll t_n\ll\frac{n^{\frac{3p-4}{4(4-p)}}}{n^{1/4}}$ we have that
$$
\sqrt{\frac{2\pi n}{p}}e^{-\frac{t_n^2nL_Z^4}{p}\big(1+o(1)\big)}\big(1+o(1)\big)\leq\Pro\left[\Bigg|\frac{\Vert X_n\Vert_2^2}{nL_Z^2}-1\Bigg|\geq t_n\right]\leq\sqrt{\frac{2\pi n}{p}}e^{-\frac{t_n^2nL_Z^4(1+o(1))}{2\Var(Z^2)}}\big(1+o(1)\big),
$$
\end{thmalpha}

As already mentioned in the introduction we look at concentration around a sphere of a slightly different radius from $n^{-1/2}\E\big[\|X_n\|_2^2\big]^{1/2}$. Thus, in order to be able to compare the bounds we obtain in the case $1\leq p <2$ with \eqref{eq:Concentration Schechtamnn-Zinn-Naor-leq2}, one should understand the speed at which $n^{-1/2}\E\big[\|X_n\|_2^2\big]^{1/2}$ converges to its limit $L_Z$.

\section{Preliminaries}

We shall present now the notation and fundamental notions used in this paper. Concerning notation, for a Borel set $A\subset \R^n$, we shall denote by $\vol_n(A)$ the $n$-dimensional Lebesgue measure of $A$. We denote by $\SSS^{n-1}:= \{\theta\in\R^n\,:\, \|\theta\|_2=1\}$ the Euclidean unit sphere in $n$-dimensional space.
In this paper, for two sequences $(x_n)_{n\in\N}$ and $(y_n)_{n\in\N}$ we use the Landau notation $x_n\in o(y_n)$ or $x_n\ll y_n$ if $\lim_{n\to\infty} \frac{x_n}{y_n}=0$. We simply write $o(1)$ to refer to some sequence tending to zero as $n\to\infty$.
Moreover, for two functions $f,g:\R\to\R$, we write $f\in\Omega(g)$ whenever there exists $C\in(0,\infty)$ and $x_0\in(0,\infty)$ such that for all $x>x_0$ one has that $C|g(x)| \leq |f(x)|$.

\subsection{Isotropic convex bodies}\label{subsec:isotropic convex bodies}

Let us start with some basics concerning the isotropic position of convex bodies, i.e., of compact and convex sets with non-empty interior. We say that a convex body $K\subset \R^n$ is isotropic (or in isotropic position) whenever the following three conditions are satisfied
\begin{itemize}
\item $\vol_n(K)=1$
\item For all $\theta\in\SSS^{n-1}$, we have $\int_{K}\langle\theta,x \rangle\,dx=0$ (centroid at the origin)
\item There exists a constant $L_K\in(0,\infty)$ such that for all $\theta\in\SSS^{n-1}$, we have $\int_{K}\langle\theta,x \rangle^2\,\dint x=L_K^2$.
\end{itemize}
$L_K$ is called the isotropic constant of $K$. Every convex body can be brought, by means of an affine transformation, into isotropic position. This affine transformation is unique up to orthogonal transformations and, since if a convex body $K$ is isotropic then so is, with the same isotropic constant, each orthogonal image of $K$, we can define the isotropic constant of any convex body as the isotropic constant of its isotropic image.

The affine map that takes a convex body $K\subset\R^n$ to an isotropic image appears as the solution of a minimization problem and then the isotropic constant of any convex body can be defined as
\[
L_K^2:= \frac{1}{n} \min\Bigg\{\frac{1}{\vol_n(TK)^{1+2/n}}\int_{a+TK} \|x\|_2^2 \,dx\,:\, a\in\R^n, T\in\text{GL}(n)\Bigg\},
\]
Here $\text{GL}(n)$ denotes the general linear group on $\R^n$. Obviously, from the definition, $L_K$ is an affine invariant. Notice that if $K$ is a $1$-symmetric convex body (i.e., invariant under permutations and change of sign in the coordinates with respect to an orthonormal basis) we have that $\widetilde{K}:=\frac{K}{\vol_n(K)^{1/n}}$ is isotropic and then
$$
nL_K^2=\frac{1}{\vol_n(K)^{1+2/n}}\int_{K} \|x\|_2^2 \,dx.
$$

We refer the reader to \cite[Chapter 2]{IsotropicConvexBodies} for more information.

\subsection{Orlicz spaces and Gibbs measures}

Let $M:\R\to[0,\infty)$ be an Orlicz function, i.e., an even and convex function such that $M(0)=0$ and $M(t)>0$ for every $t>0$. For $R\in(0,\infty)$ and $n\in\N$, let us denote by $B_M^n(nR)$ the Orlicz ball
$$
B_M^n(nR):=\Big\{x\in\R^n\,:\,\sum_{i=1}^nM(x_i)\leq nR\Big\}.
$$
We refer to \cite[Lemma 2.1]{KP2020} to see that $B_M^n(1)$ coincides with the one defined as the unit ball of the Luxemburg norm $\|(x_1,\dots,x_n)\|_M = \inf\big\{\rho>0\,:\, \sum_{i=1}^n M(|x_i|/\rho) \leq 1 \big\}$ on $\R^n$. Observe that since $M$ is even, from the definition of $B_M^n(nR)$  we have that any Orlicz ball is a $1$-symmetric convex body and then $\widetilde{B}_M^n(nR)$ is an isotropic convex body.

The precise asymptotic volume of Orlicz balls has recently been obtained in \cite{KP2020}. For our purpose knowing the asymptotic logarithmic volume will be enough and as \cite[Theorem A]{KP2020} states, this is given as follows: let $n\in\N$, $R\in(0,\infty)$, and $M$ be an Orlicz function. Then, as $n\to\infty$,
\begin{align}\label{eq:log-volume orlicz}
\vol_n\big(B_M^n(nR)\big)^{1/n} \to e^{\varphi(\alpha_*)-\alpha_* R}.
\end{align}

Given an Orlicz function $M:\R\to[0,\infty)$ and $R\in(0,\infty)$, let $\varphi_M:(-\infty,0)\to\R$ be the function
$$
\varphi_M(\alpha)=\log \int_\R e^{\alpha M(x)}dx,\quad\alpha<0.
$$
In statistical mechanics and large deviation parlance, $\varphi$ is the logarithm of the partition function with potential $M$ at inverse temperature $-\alpha$, $\alpha\in(-\infty,0)$.
Notice that $\varphi$ is strictly increasing and convex (as a consequence of H\"older's inequality). Besides,
\begin{itemize}
\item $\displaystyle{\varphi_M^\prime(\alpha)}=\frac{\int_\R M(x)e^{\alpha M(x)}dx}{\int_\R e^{\alpha M(x)}dx}$,
\item $\displaystyle{\lim_{\alpha\to-\infty}\varphi_M^\prime(\alpha)=0}$,
\item $\displaystyle{\lim_{\alpha\to0^-}\varphi_M^\prime(\alpha)=\infty}$,
\item $\varphi_M^\prime$ is continuous and increasing (since $\varphi$ is convex).
\end{itemize}
For proofs of the second and third property, we refer the reader to the arXiv version of \cite{KP2020} or \cite[Theorem 6.2]{D1954}.
The previous properties imply that there exists a unique $\alpha_*\in(-\infty,0)$ such that $\varphi_M^\prime(\alpha_*)=R$. Let $p_M:\R\to[0,\infty)$ be the log-concave Gibbs density with potential $M$ given by
$$
p_M(x)=e^{\alpha_*M(x)-\varphi_M(\alpha_*)}=\frac{ e^{\alpha_* M(x)}}{\int_\R e^{\alpha_* M(x)}dx},\quad x\in\R,
$$
and let $Z$ be a random variable with density $p$ with respect to Lebesgue measure. Then we have the following
\begin{itemize}
\item $\E[M(Z)]=\varphi_M^\prime(\alpha_*)=R$
\item $\Var[M(Z)]=\varphi_M^{\prime\prime}(\alpha_*)$.
\end{itemize}
Moreover, it was shown in the proof of \cite[Proposition 3.2]{KP2020} (see Equation \eqref{eq:asymptotic expectation} below)
that for the random variable $Y:=M(Z)-R$ and a sequence $(Y_i)_{i\in\N}$ of independent copies of $Y$, one has the asymptotic formula
\begin{align}\label{eq:asymptotic expectation}
\E\left[\chi_{(-\infty,0]}\left(\sum_{i=1}^nY_i\right)e^{-\alpha_*\sum_{i=1}^nY_i}\right]=\frac{1+o(1)}{|\alpha_*|\sqrt{2\pi n\varphi^{\prime\prime}(\alpha_*)}}.
\end{align}

\subsection{Bernstein's inequality, LDPs, and MDPs}

Bernstein's inequality gives an estimate for the probability of the sum of $n$ independent copies of a centered random variable being, in absolute value, larger than $nt$. There are different versions of this inequality, depending on the assumptions on the random variable, which vary in the range of $t$ in which the estimate is valid. We state two of them in the following theorem, which we take from \cite[Proposition 1, (ii) and (iii)]{BLM1988}.
\begin{thm}[Bernstein's inequality]\label{thm:Bernstein}
Let $n\in\N$ and $(Y_i)_{i=1}^n$ be a sequence of independent copies of a centered random variable $Y$. Then
\begin{enumerate}
\item[i)] If there exist $\lambda\in(0,\infty)$ such that $\E \big[e^{|Y|/\lambda}\big]<\infty$ and $A\in(0,\infty)$ such that $\inf\big\{\lambda>0\,:\,\E \big[e^{\frac{|Y|}{\lambda}}\big]\leq 2\big\}\leq A$, then
$$
\Pro\left[\Bigg|\frac{1}{n}\sum_{i=1}^n Y_i\Bigg|> t\right]\leq 2e^{-\frac{t^2n}{16 A^2}}\quad\forall t\in(0,4A).
$$
\item[ii)] If there exist $\lambda\in(0,\infty)$ such that $\E\big[ e^{|Y|^2/\lambda^2}\big]<\infty$ and $A\in(0,\infty)$ such that $\inf\big\{\lambda>0\,:\,\E\big[ e^{\frac{|Y|^2}{\lambda^2}}\big]\leq 2\big\}\leq A$, then
$$
\Pro\left[\Bigg|\frac{1}{n}\sum_{i=1}^n Y_i\Bigg|> t\right]\leq 2e^{-\frac{t^2n}{8 A^2}}\quad\forall t>0.
$$
\end{enumerate}
\end{thm}
\begin{rmk}
Notice that while the estimate in ii), which assumes a stronger condition on the random variables $Y_i$ than the first , is valid for every $t>0$, the estimate in i) is valid only for small values of $t$.
\end{rmk}
In the large deviations theory, Cramer's \cite{C1938} theorem gives us the asymptotic sharp value of the constant in the exponent for every fixed $t$.

\begin{df}
Let $(X_n)_{n\in\N}$ be a sequence of random vectors taking values in $\R^d$. Further, let $s:\N\to[0,\infty]$ and $I:\R^d\to[0,\infty]$ be a lower semi-continuous function with compact level sets $\{x\in\R^d\,:\, I(x) \leq \alpha \}$, $\alpha\in\R$. We say that $(X_n)_{n\in\N}$ satisfies a large deviation principle (LDP) with speed $s(n)$ and (good) rate function $I$ if
$$
-\inf_{x\in A^\circ}I(x) \leq\liminf_{n\to\infty}{1\over s(n)}\log\Pro[X_n\in A]\leq\limsup_{n\to\infty}{1\over s(n)}\log \Pro[X_n\in A]\leq-\inf_{x\in\overline{A}}I(x)
$$
for every Lebesgue measurable set $A\in\R^d$.
\end{df}

We notice that on the class of all $I$-continuity sets, that is, on the class of Lebesgue measurable sets $A$ for which $I(A^\circ)=I(\bar{A})$ with $I(A):=\inf\{I(x):x\in A\}$, one has the exact limit relation
$$
\lim_{n\to\infty}{1\over s(n)}\log \Pro[X^{(n)}\in A]=-I(A).
$$

The following version of Cram\'er's theorem was taken from \cite[Corollary 2.2.19]{DZ2010}.

\begin{thm}[Cram\'er]\label{thm:Cramer}
Let $(Y_n)_{n=1}^\infty$ be a sequence of independent copies of a centered random variable $Y$. Assume that
$$
\Lambda(u)=\log \E \big[e^{uY}\big] <\infty
$$
in a neighborhood of $0$. Then, for every $\in(0,\infty)$,
$$
\lim_{n\to\infty}\frac{\log\Pro\big[\left|\frac{1}{n}\sum_{i=1}^n Y_i\right|\geq t\big]}{n}=-\inf_{|s|\geq t}\Lambda^*(s),
$$
where $\Lambda^*$ is the Legendre transform of $\Lambda$.
\end{thm}


If instead of considering a fixed value of $t$ we consider a sequence $(t_n)_{n=1}^\infty$ converging to $0$ we turn our look to moderate deviation principles, which are nothing but large deviation principles under a different normalization. In this case, the rate function, which gives the asymptotic value of the constant in the exponentially decreasing probability is Gaussian under some assumptions on the random variable. The following theorem can be found in \cite[Theorem 2.2]{EL2003}.
\begin{thm}[Eichelsbacher-L\"owe]\label{lem:eicheslbacher-loewe}
Let $(Y_n)_{n\in\N}$ be a sequence of  independent copies of a centered random variable $Y$ with positive variance and let $(s_n)_{n\in\N}$ be a sequence of positive real numbers such that $1\ll s_n\ll\sqrt{n}$. Assume that
\begin{align}\label{eq:limit condition eichelsbacher-loewe}
\lim_{n\to\infty} \frac{1}{s_n^2} \log\Big(n\,\Pro\big[|Y| > \sqrt{n}s_n\big]\Big) & = -\infty.
\end{align}
Then $(\frac{1}{s_n\sqrt{n}}\sum_{i=1}^n Y_i)_{n\in\N}$ satisfies an MDP on $\R$ with speed $s_n^2$ and good rate function $I:\R\to[0,\infty)$ given by $I(x)={x^2\over 2\Var[Y]}$. In particular, for every $t\in(0,\infty)$,
$$
\lim_{n\to\infty}\frac{\log\Pro\left[\left|\frac{1}{s_n\sqrt{n}}\sum_{i=1}^n Y_i\right|\geq t\right]}{s_n^2}=-\frac{t^2}{2\Var[Y]}.
$$
\end{thm}

The following theorem gives an MDP for sums of independent and identically distributed random vectors under similar conditions to the ones in Cram\'er's theorem and is due to Petrov \cite{P1975} (see also \cite[Theorem 3.7.1]{DZ2010}).

\begin{thm}[MDP for sums of i.i.d. random vectors]\label{lem:MDPVector}
Let $(Y_n)_{n\in\N}$ be a sequence of independent copies of a centered random vector $Y$ in $\R^d$ and let $(s_n)_{n\in\N}$ be sequence of positive real numbers such that $1\ll s_n\ll\sqrt{n}$. We assume that $Y$ is centered, its covariance matrix $\bC=\Cov(Y)$ is invertible, and that
$$
\Lambda(u):=\log\E[e^{\langle u,Y\rangle}]<\infty
$$
for every $u$ in a neighborhood of $0$. Then the sequence of random vectors $\frac{1}{s_n\sqrt{n}}\sum_{i=1}^nY_i$, $n\in\N$, satisfies an LDP (i.e.,\ an MDP as the sum is scaled by $s_n\sqrt{n}$) with speed $s_n^2$  and rate function $I(x)={1\over 2}\langle x,\bC^{-1}x\rangle$, $x\in\R^d$.
\end{thm}

Next, assume that a sequence $(X_n)_{n\in\N}$ of random variables satisfies an LDP with speed $s_n$ and rate function $I$. Suppose now that $(Y_n)_{n\in\N}$ is a sequence of random variables that are `close' to the ones from the first sequence. The next result provides conditions under which in such a situation an LDP from the first can be transferred to the second sequence, see \cite[Theorem 4.2.13]{DZ2010}.

\begin{lemma}[Exponential equivalence]\label{lem:exponentially equivalent}
Let $(X_n)_{n\in\N}$ and $(Y_n)_{n\in\N}$ be two sequences of random vectors in $\R^d$ and assume that $(X_n)_{n\in\N}$ satisfies an LDP on $\R^d$ with speed $s_n$ and rate function $I$. Further, suppose that the two sequences $(X_n)_{n\in\N}$ and $(Y_n)_{n\in\N}$ are exponentially equivalent, i.e.,
$$
\limsup_{n\to\infty}s_n^{-1}\log\Pro\big[\|X_n-Y_n\|_2>\delta\big] = -\infty
$$
for any $\delta\in(0,\infty)$. Then $(Y_n)_{n\in\N}$ satisfies an LDP on $\R^d$ with the same speed and the same rate function.
\end{lemma}

\section{The asymptotic thin-shell width concentration for Orlicz balls}

We start with a few technical preparations. The first result relates a thin-shell estimate for points chosen uniformly at random from an Orlicz ball with tail bounds for a modified Gibbs distribution.

\begin{lemma}\label{lem:WidthAnyt}
Let $M:\R\to[0,\infty)$ be an Orlicz function, $R\in(0,\infty)$, $n\in\N$. Let $\varphi_M:(-\infty,0)\to\R$ be the log-partition function with potential $M$, and let $(Z_i)_{i\in\N}$ be a sequence of independent random variables with Gibbs density $p_M$, i.e.,
$$
\varphi_M(\alpha)=\log \int_\R e^{\alpha M(x)}dx\quad\textrm{and}\quad p_M(x)=\frac{e^{\alpha_* M(x)}}{ \int_\R e^{\alpha_* M(x)}dx}.
$$
where $\alpha_*\in(-\infty,0)$ is chosen such that $\varphi_M^\prime(\alpha_*)=R$. Let $Y_i^{(2)}:=Z_i^2-L_Z^2$, $i\in\N$, where
$$
L_Z^2:=\frac{\int_\R x^2e^{\alpha_* M(x)}dx}{\int_\R e^{\alpha_* M(x)}dx}.
$$
Then, if $X_n$ is a random vector uniformly distributed on $B_M^n(nR)$, we have, for every $t\in(0,\infty)$ that, as $n\to\infty$,
$$
\Pro\left[\Bigg|\frac{\Vert X_n\Vert_2^2}{n}-L_Z^2\Bigg|\geq t\right]\leq|\alpha_*|\sqrt{2\pi n\,\varphi_M^{\prime\prime}(\alpha_*)}\,\Pro\left[\Bigg|\frac{1}{n}\sum_{i=1}^n Y_i^{(2)}\Bigg|\geq t\right]\big(1+o(1)\big),
$$
where the sequence $o(1)$ does not depend on $t$.
\end{lemma}
\begin{proof}
Let $(Z_i)_{i=1}^n$ be independent identically distributed copies of a (symmetric) random variable $Z$ with Gibbs density $p_M$. For any $1\leq i\leq n$ consider the centered random variables $Y_i^{(1)}$ and $Y_i^{(2)}$ defined by
\[
Y_i^{(1)}=M(Z_i)-R \qquad \text{and} \qquad Y_i^{(2)}=Z_i^2-L_Z^2,
\]
where $R\in(0,\infty)$ and
\[
  L_Z^2=\E[Z^2]=\int_\R x^2p_M(x)\,dx=\frac{\int_\R x^2e^{\alpha_* M(x)}dx}{\int_\R e^{\alpha_* M(x)}dx}.
\]
If $X_n$ is a random vector uniformly distributed on $B_M^n(nR)$, then, for any $t>0$, we have
\begin{eqnarray*}
&&\Pro\left[\Bigg|\frac{\Vert X_n\Vert_2^2}{n}-L_Z^2\Bigg|\geq t\right]=\frac{\int_{\R^n}\chi_{B_M^n(nR)}(x)\chi_{\R\setminus(L_Z^2-t,L_Z^2+t)}\left(\frac{1}{n}\Vert x\Vert_2^2\right)dx}{\int_{\R^n}\chi_{B_M^n(nR)}(x)dx}\cr
&=&\frac{\int_{\R^n}\chi_{B_M^n(nR)}(x)\chi_{\R\setminus(n(L_Z^2-t),n(L_Z^2+t))}\left(\Vert x\Vert_2^2\right)e^{-\alpha_*\sum_{i=1}^nM(x_i)+n\varphi(\alpha_*)}\prod_{i=1}^np_M(x_i)dx}{\int_{\R^n}\chi_{B_M^n(nR)}(x)e^{-\alpha_*\sum_{i=1}^nM(x_i)+n\varphi(\alpha_*)}\prod_{i=1}^np_M(x_i)dx}\cr
&=&\frac{\E\left[\chi_{B_M^n(nR)}\left((Z_1,\dots,Z_n)\right)\chi_{\R\setminus(n(L_Z^2-t),n(L_Z^2+t))}\left(\Vert (Z_1,\dots, Z_n)\Vert_2^2\right)e^{-\alpha_*\sum_{i=1}^nM(Z_i)}\right]}{\E\left[\chi_{B_M^n(nR)}\left((Z_1,\dots,Z_n)\right)e^{-\alpha_*\sum_{i=1}^nM(Z_i)}\right]}\cr
&=&\frac{\E\left[\chi_{(-\infty,0]}\left(\sum_{i=1}^nY_i^{(1)}\right)\chi_{\R\setminus(-nt,nt)}\left(\sum_{i=1}^n Y_i^{(2)}\right)e^{-\alpha_*\sum_{i=1}^nY_i^{(1)}}\right]}{\E\left[\chi_{(-\infty,0]}\left(\sum_{i=1}^nY_i^{(1)}\right)e^{-\alpha_*\sum_{i=1}^nY_i^{(1)}}\right]}.
\end{eqnarray*}

Since $\alpha_*\in(-\infty,0)$, we have that
\begin{eqnarray*}
&&\E\left[\chi_{(-\infty,0]}\left(\sum_{i=1}^nY_i^{(1)}\right)\chi_{\R\setminus(-nt,nt)}\left(\sum_{i=1}^n Y_i^{(2)}\right)e^{-\alpha_*\sum_{i=1}^nY_i^{(1)}}\right]\cr
&\leq&\E\left[\chi_{(-\infty,0]}\left(\sum_{i=1}^nY_i^{(1)}\right)\chi_{\R\setminus(-nt,nt)}\left(\sum_{i=1}^n Y_i^{(2)}\right)\right]\cr
&\leq&\E\left[\chi_{\R\setminus(-nt,nt)}\left(\sum_{i=1}^n Y_i^{(2)}\right)\right]=\Pro\left[\Bigg|\frac{1}{n}\sum_{i=1}^n Y_i^{(2)}\Bigg|\geq t\right].
\end{eqnarray*}
The proof of \cite[Proposition 3.2]{KP2020} (see Equation \eqref{eq:asymptotic expectation} above)
yields that, as $n\to\infty$,
\[
\E\left[\chi_{(-\infty,0]}\left(\sum_{i=1}^nY_i^{(1)}\right)e^{-\alpha_*\sum_{i=1}^nY_i^{(1)}}\right]=\frac{1+o(1)}{|\alpha_*|\sqrt{2\pi n\varphi_M^{\prime\prime}(\alpha_*)}}
\]
and thus, as $n\to\infty$,
\[
\Pro\left[\Bigg|\frac{\Vert X_n\Vert_2^2}{n}-L_Z^2\Bigg|\geq t\right]\leq|\alpha_*|\sqrt{2\pi n\varphi_M^{\prime\prime}(\alpha_*)}\,\Pro\left[\Bigg|\frac{1}{n}\sum_{i=1}^n Y_i^{(2)}\Bigg|\geq t\right]\big(1+o(1)\big).
\]
This completes the proof.
\end{proof}

The following lemma establishes a thin-shell concentration estimate around a sphere of radius $\sqrt{n}L_Z$. As we shall see later, it can be improved under some growth assumptions on the Orlicz function $M$.

\begin{lemma}\label{lem: thin-shell concetration via chebyshev}
Let $M:\R\to[0,\infty)$ be an Orlicz function, $n\in\N$, $R\in(0,\infty)$, and $\varphi_M:(-\infty,0)\to\R$ be the log-partition function with potential $M$, and let $(Z_i)_{i\in\N}$ be a sequence of independent random variables with Gibbs density $p_M$, i.e.,
$$
\varphi_M(\alpha)=\log \int_\R e^{\alpha M(x)}dx\quad\textrm{and}\quad p_M(x)=\frac{e^{\alpha_* M(x)}}{ \int_\R e^{\alpha_* M(x)}dx},
$$
where $\alpha_*\in(-\infty,0)$ is chosen such that $\varphi_M^\prime(\alpha_*)=R$.
Let $Y_i^{(2)}:=Z_i^2-L_Z^2$, $i\in\N$, where
$$
L_Z^2:=\frac{\int_\R x^2e^{\alpha_* M(x)}dx}{\int_\R e^{\alpha_* M(x)}dx}.
$$
Then, if $X_n$ is a random vector uniformly distributed on $B_M^n(nR)$, we have, for every $t\in(0,\infty)$ that, as $n\to\infty$,
$$
\Pro\left[\Bigg|\frac{\Vert X_n\Vert_2^2}{n}-L_Z^2\Bigg|\geq t\right]\leq\frac{|\alpha_*|\,\Var[Y_1^{(2)}]\sqrt{2\pi\varphi_M^{\prime\prime}(\alpha_*)}}{t^2\sqrt{n}}\big(1+o(1)\big),
$$
where the sequence $o(1)$ does not depend on $t$.
\end{lemma}
\begin{proof}
Since $M$ is an Orlicz function, $(0,\infty)\ni t \mapsto M(t)/t$ is monotone increasing. Therefore, it is easy to see that $\Var[Y_1^{(2)}] \leq \E[Z_1^4]+L_Z^4<+\infty$. This means that we can apply Chebyshev's inequality, which, combined with the independence of the random variables $Y_i^{(2)}$, yields for any $t\in(0,\infty)$ that
\[
\Pro\Bigg[\Bigg|\sum_{i=1}^n Y_i^{(2)}\Bigg|\geq nt\Bigg] = \Pro\Bigg[\Bigg(\sum_{i=1}^n Y_i^{(2)}\Bigg)^2\geq n^2t^2\Bigg]
 \leq \frac{n\Var[Y_1^{(2)}]}{n^2t^2}=\frac{\Var[Z^2]}{nt^2}.
\]
Therefore, for any $t\in(0,\infty)$ we have that, as $n\to\infty$,
$$
\Pro\left[\Bigg|\frac{\Vert X_n\Vert_2^2}{n}-L_Z^2\Bigg|\geq t\right]\leq\frac{|\alpha_*|\Var[Y_1^{(2)}]\sqrt{2\pi\varphi_M^{\prime\prime}(\alpha_*)}}{t^2\sqrt{n}}\big(1+o(1)\big),
$$
where the sequence $o(1)$ does not depend on $t$. This completes the proof.
\end{proof}

We are now going to prove Theorem \ref{thm:concentration thin-shell orlicz}. Compared to the previous lemma, we add a growth condition on the Orlicz function $M$ so that we can use a result on moderate deviation principles for sums of independent random variables due to Petrov \cite{P1975}. This allows us to improve significantly upon the bound presented in Lemma \ref{lem: thin-shell concetration via chebyshev}. The mentioned growth condition is in particular satisfied by all $2$-convex Orlicz functions, i.e., those for which $M(\sqrt{\cdot})$ is a convex function.

\begin{proof}[Proof of Theorem \ref{thm:concentration thin-shell orlicz}]
Since $M\in\Omega(x^2)$ as $x\to\infty$, we have that if $Y_i^{(2)}=Z_i^2-L_Z^2$ for every $i\in\N$, then
$$
\E\big[e^{uY_1^{(2)}}\big]=e^{-uL_Z^2}\E \big[e^{uZ_1^2}\big]=\frac{e^{-uL_Z^2}\int_{\R}e^{ux^2e^{\alpha_*M(x)}}dx}{\int_\R e^{\alpha_*M(x)}dx} <\infty
$$
for all $u$ in a neighborhood of $0$. By Petrov's moderate deviations result from \cite{P1975} (which we apply now in the form of Theorem \ref{lem:MDPVector} in dimension $d=1$), 
we have that if $(s_n)_{n\in\N}\in(0,\infty)^\N$ is a sequence such that $1\ll s_n\ll \sqrt{n}$, then the sequence of random variables
$$
\Big(\frac{1}{s_n\sqrt{n}}\sum_{i=1}^n Y_i^{(2)}\Big)_{n\in\N}
$$
satisfies an MDP with speed $s_n^2$ and rate function $I:\R\to[0,\infty]$, $I(x):=\frac{x^2}{2\Var[Y_1^{(2)}]}=\frac{x^2}{2\Var[Z^2]}$. Using this with $s_n=t_n\sqrt{n}$ (which indeed means that $1\ll s_n\ll \sqrt{n}$) and the fact that $(-1,1)^c=\R\setminus (-1,1)$ is an $I$-continuity set, we get
\begin{align*}
\lim_{n\to\infty} \frac{1}{s_n^2}\log\, \Pro\left[\Bigg|\frac{1}{n}\sum_{i=1}^n Y_i^{(2)}\Bigg|\geq t_n\right] & = \lim_{n\to\infty} \frac{1}{s_n^2}\log\, \Pro\left[\Bigg|\frac{1}{s_n\sqrt{n}}\sum_{i=1}^n Y_i^{(2)}\Bigg|\geq 1\right] = -I(1).
\end{align*}
As $n\to\infty$, this translates to
\begin{eqnarray*}
\Pro\left[\Bigg|\frac{1}{n}\sum_{i=1}^n Y_i^{(2)}\Bigg|\geq t_n\right]
=e^{-I(1)s_n^2\,(1+o(1))}
= e^{-\frac{t_n^2n(1+o(1))}{2\Var[Z^2]}}.
\end{eqnarray*}
Putting everything together, we obtain, as $n\to\infty$,
$$
\Pro\left[\Bigg|\frac{\Vert X_n\Vert_2^2}{n}-L_Z^2\Bigg|\geq t_n\right]\leq|\alpha_*|\sqrt{2\pi n\,\varphi_M^{\prime\prime}(\alpha_*)}e^{-\frac{t_n^2n(1+o(1))}{2\Var[Z^2]}}\big(1+o(1)\big)
$$
or, equivalently,
$$
\Pro\left[\Bigg|\frac{\Vert X_n\Vert_2^2}{nL_Z^2}-1\Bigg|\geq t_n\right]\leq|\alpha_*|\sqrt{2\pi n\,\varphi_M^{\prime\prime}(\alpha_*)}e^{-\frac{t_n^2nL_Z^4(1+o(1))}{2\Var[Z^2]}}\big(1+o(1)\big),
$$
which completes the proof.
\end{proof}

\begin{rmk}
We see from the proof of Theorem \ref{thm:concentration thin-shell orlicz} that for each sequence $\frac{1}{\sqrt{n}}\ll t_n\ll1$, the constant in the exponential, $\frac{1+o(1)}{\Var (Z^2)}$, is asymptotically sharp for the estimate of $\Pro\left[\left|\frac{1}{n}\sum_{i=1}^n Y_i^{(2)}\right|\geq t_n\right]$. However, the sequence $o(1)$ in the exponent depends on the sequence $(s_n)_{n\in\N}$ and so on $(t_n)_{n=1}^\infty$.  Alternatively, since $M\in\Omega(x^2)$ as $x\to\infty$, we have that there exists $\lambda>0$ such that
$$
\E\big[e^{Y_1^{(2)}/\lambda}\big]=e^{-\frac{L_Z^2}{\lambda}}\E\big[ e^{Z^2/\lambda}\big]=\frac{e^{-\frac{L_Z^2}{\lambda}}\int_{\R}e^{\frac{x^2}{\lambda}}e^{\alpha_*M(x)}dx}{\int_\R e^{\alpha_*M(x)}dx} <\infty.
$$
and then calling $A:=\inf\Big\{\lambda>0\,:\,\E \big[e^{Y_1^{(2)}/\lambda}\big]\leq2\Big\}$  we have by Bernstein's inequality (see Theorem \ref{thm:Bernstein} (i)) that for any $t\in(0,4A)$,
$$
\Pro\left[\Bigg|\frac{1}{n}\sum_{i=1}^n Y_i^{(2)}\Bigg|\geq t\right]\leq 2e^{-\frac{t^2n}{16 A^2}}.
$$
Therefore, by Lemma \ref{lem:WidthAnyt}, for any $t\in(0,4A)$,
$$
\Pro\left[\Bigg|\frac{\Vert X_n\Vert_2^2}{n}-L_Z^2\Bigg|\geq t\right]\leq2|\alpha_*|\sqrt{2\pi n\,\varphi_M^{\prime\prime}(\alpha_*)}e^{-\frac{t^2n}{16A^2}}\big(1+o(1)\big)
$$
or, equivalently, for any $t\in(0,4AL_Z^{-2})$
$$
\Pro\left[\Bigg|\frac{\Vert X_n\Vert_2^2}{nL_Z^2}-1\Bigg|\geq t\right]\leq2|\alpha_*|\sqrt{2\pi n\,\varphi_M^{\prime\prime}(\alpha_*)}e^{-\frac{t^2nL_Z^4}{16A^2}}\big(1+o(1)\big),
$$
where the sequence $o(1)$ does not depend on $t$.

Alernatively, instead of Bernstein's inequality we can make use of Cram\'er's large deviation result (see Theorem \ref{thm:Cramer}) to obtain an upper bound for every $t\in(0,\infty)$ at the cost of having a sequence $o(1)$ depending on $t$ in the exponential.
\end{rmk}

\begin{rmk}\label{rmk:RemoveSquare}
While thin-shell concentration inequalities are usually stated for the random variable $\frac{\Vert X_n\Vert_2}{\sqrt n}$ conveniently normalized, in the above result we obtained a thin-shell concentration inequality for the random variable $\frac{\Vert X_n\Vert_2^2}{nL_Z^2}$. An inequality for the random variable $\frac{\Vert X_n\Vert_2}{\sqrt nL_Z}$ is immediately obtained since for any $t>0$, as $\frac{\Vert X_n\Vert_2}{\sqrt nL_Z}\geq0$, we have that
$$
\Pro\left[\Bigg|\frac{\Vert X_n\Vert_2}{\sqrt{n}L_Z}-1\Bigg|\geq t\right]=\Pro\left[\Bigg|\frac{\Vert X_n\Vert_2^2}{nL_Z^2}-1\Bigg|\geq t\Bigg|\frac{\Vert X_n\Vert_2}{\sqrt{n}L_Z}+1\Bigg|\right]\leq\Pro\left[\Bigg|\frac{\Vert X_n\Vert_2^2}{nL_Z^2}-1\Bigg|\geq t\right],
$$
\end{rmk}


Notice that while Theorem \ref{thm:concentration thin-shell orlicz} shows that in high dimensions if $X_n$ is a random vector uniformly distributed on $B_M^n(nR)$, then $\frac{\Vert X_n\Vert_2}{\sqrt{n}}$ concentrates around $L_Z$, which is (as we will see) the asymptotic value of $\left(\E\left[\frac{\Vert X_n\Vert^2}{n}\right]\right)^{1/2}$, the thin-shell width conjecture is a conjecture on the concentration of $\frac{\Vert X_n\Vert_2}{\sqrt{n}}$ around $\left(\E\left[\frac{\Vert X_n\Vert^2}{n}\right]\right)^{1/2}$. We do not know if, in general, $\left(\E\left[\frac{\Vert X_n\Vert^2}{n}\right]\right)^{1/2}$ converges fast enough to $L_Z$ so that estimates on the concentration around $L_Z$ imply the same asymptotic estimates on the concentration around $\left(\E\left[\frac{\Vert X_n\Vert^2}{n}\right]\right)^{1/2}$. However, the next result shows that under some extra condition on the growth of the function $M$, we can transfer the concentration around $L_Z$ to concentration  around $\left(\E\left[\frac{\Vert X_n\Vert^2}{n}\right]\right)^{1/2}$, like in \eqref{eq:Concentration Guedon-Milman}. Furthermore, in such case the estimate \eqref{eq:Concentration Lee-Vempala} is also improved for large values of $t$.


\begin{thm}\label{thm:concetration thin-shell orlicz around expectation}
Let $M:\R\to[0,\infty)$ be an Orlicz function such that $M(x)=\Omega(x^4)$ as $x\to\infty$, $R>0$, $n\in\N$, $\varphi_M:(-\infty,0)\to\R$ the log-partition function with potential $M$ and let $(Z_i)_{i\in\N}$ be a sequence of independent random variables with Gibbs density $p_M$, i.e.,
$$
\varphi_M(\alpha)=\log \int_\R e^{\alpha M(x)}dx\quad\textrm{and}\quad p_M(x)=\frac{e^{\alpha_* M(x)}}{ \int_\R e^{\alpha_* M(x)}dx},
$$
where $\alpha_*\in(-\infty,0)$ is chosen such that $\varphi_M^\prime(\alpha_*)=R$. Let $Y_i^{(2)}:=Z_i^2-L_Z^2$, $i\in\N$, where
$$
L_Z^2:=\frac{\int_\R x^2e^{\alpha_* M(x)}dx}{\int_\R e^{\alpha_* M(x)}dx}.
$$

Let $X_n$ be a random vector uniformly distributed on $B_M^n(nR)$. Then, for any sequence $(t_n)_{n=1}^\infty \in(0,\infty)^{\N}$, we have that, as $n\to\infty$,
$$
\Pro\left[\Bigg|\frac{\Vert X_n\Vert_2}{\left(\E\Vert X_n\Vert_2^2\right)^{1/2}}-1\Bigg|\geq t_n\right]\leq  4|\alpha_*|\sqrt{2\pi n\,\varphi^{\prime\prime}(\alpha_*)}e^{-\frac{t_n^2nL_Z^4(1+o(1))}{8A^2}}\big(1+o(1)\big),
$$
where the sequences $o(1)$ do not depend on $(t_n)_{n=1}^\infty$ and $A:=\inf\Big\{\lambda>0\,:\,\E\Big[ \exp\Big(\frac{(Y_1^{(2)})^2}{\lambda^2}\Big)\Big] \leq2\Big\}$.
\end{thm}

\begin{proof}
First we observe that by Lemma \ref{lem:WidthAnyt}, if $X_n$ and $\overline{X}_n$ are two independent random vectors uniformly distributed on $B_M^n(nR)$ then, for every $t\in(0,\infty)$, we obtain, as $n\to\infty$,
\begin{eqnarray}\label{eq:tail bound x and x bar}
&&\Pro\left[\Bigg|\frac{\Vert X_n\Vert_2^2}{n}-\frac{\Vert \overline{X}_n\Vert_2^2}{n}\Bigg|>t\right]
\leq\Pro\left[\Bigg|\frac{\Vert X_n\Vert_2^2}{n}-L_Z^2\Bigg|>t\right]+\Pro\left[\Bigg|\frac{\Vert \overline{X}_n\Vert_2^2}{n}-L_Z^2\Bigg|>t\right]\cr
&\leq&2|\alpha_*|\sqrt{2\pi n\varphi_M^{\prime\prime}(\alpha_*)}\,\Pro\left[\Bigg|\frac{1}{n}\sum_{i=1}^n Y_i^{(2)}\Bigg|\geq t\right]\big(1+o(1)\big),
\end{eqnarray}
where the sequence $o(1)$ does not depend on $t$. Using \eqref{eq:tail bound x and x bar}, we obtain that for any sequence  $(\lambda_n)_{n=1}^\infty$, as $n\to\infty$,
\begin{align}\label{eq:expectation bound lambda n}
\E \Bigg[e^{\lambda_n\left|\frac{\Vert X_n\Vert_2^2}{n}-\frac{\Vert \overline{X}_n\Vert_2^2}{n}\right|^2 }\Bigg]
& \leq 2|\alpha_*|\sqrt{2\pi n\varphi_M^{\prime\prime}(\alpha_*)}\,\E\Big[e^{\lambda_n\left|\frac{1}{n}\sum_{i=1}^n Y_i^{(2)}\right|^2}\Big]\big(1+o(1)\big).
\end{align}

Since $M\in\Omega(x^4)$ as $x\to\infty$, there exists $\lambda>0$ (large enough) such that
$$
\E \big[e^{(Y_1^{(2)})^2/\lambda^2}\big]=\frac{\int_{\R}e^{(x^2-L_Z^2)^2/\lambda^2}e^{\alpha_*M(x)}dx}{\int_\R e^{\alpha_*M(x)}dx} <\infty.
$$
Therefore, we can apply Bernstein's inequality (see Theorem \ref{thm:Bernstein} (ii)) and, letting $A:=\inf\big\{\lambda>0\,:\,\E\big[ e^{(Y_1^{(2)})^2/\lambda^2}\big]\leq 2 \big\}$, we have that, for any $t\geq0$,
$$
\Pro\left[\Bigg|\frac{1}{n}\sum_{i=1}^n Y_i^{(2)}\Bigg|\geq t\right]\leq2e^{-\frac{t^2n}{8A^2}}.
$$
Thus, choosing $\lambda_n:=\frac{n}{16A^2}$, we obtain from the previous estimate that
\begin{eqnarray}\label{eq:expectation exp lambda n bounded by 2}
\E \Big[e^{\lambda_n\left|\frac{1}{n}\sum_{i=1}^n Y_i^{(2)}\right|^2}\Big]
&=& \int_0^\infty 2\lambda_nte^{\lambda_nt^2}\Pro\left[\Bigg|\frac{1}{n}\sum_{i=1}^n Y_i^{(2)}\Bigg|\geq t\right]dt \cr
&\leq& 2\int_0^\infty 2\lambda_nte^{\lambda_nt^2}e^{-\frac{t^2n}{8A^2}}dt
=2\int_0^\infty 2\lambda_nte^{-\lambda_nt^2}dt=2.
\end{eqnarray}
Combining this bound with \eqref{eq:expectation bound lambda n} we obtain that for $\lambda_n=\frac{n}{16A^2}$, as $n\to\infty$,
$$
\E\Bigg[e^{\lambda_n\left|\frac{\Vert X_n\Vert_2^2}{n}-\frac{\Vert \overline{X}_n\Vert_2^2}{n}\right|^2 }\Bigg]\leq 4|\alpha_*|\sqrt{2\pi n\varphi_M^{\prime\prime}(\alpha_*)}\big(1+o(1)\big).
$$
Since for any $\lambda_n>0$, and any $x_0\in\R$ the function $e^{\lambda_n|x_0-\cdot|^2}$ is convex we have, by Jensen's inequality,
$$
\E \Bigg[e^{\lambda_n\left|\frac{\Vert X_n\Vert_2^2}{n}-\frac{\E \Vert X_n\Vert_2^2}{n}\right|^2}\Bigg] \leq \E \Bigg[e^{\lambda_n\left|\frac{\Vert X_n\Vert_2^2}{n}-\frac{\Vert \overline{X}_n\Vert_2^2}{n}\right|^2}\Bigg].
$$
As $n\to\infty$, we obtain from Markov's inequality combined with the previous estimate and \eqref{eq:expectation bound lambda n} that, for any sequence $(t_n)_{n=1}^\infty\in(0,\infty)^\N$,
\begin{eqnarray*}
&&\Pro\left[\Bigg|\frac{\Vert X_n\Vert_2^2}{n}-\frac{\E\big[\Vert X_n\Vert_2^2\big]}{n}\Bigg|\geq t_n\right]
=\Pro\left[e^{\lambda_n\Big|\frac{\Vert X_n\Vert_2^2}{n}-\frac{\E\Vert X_n\Vert_2^2}{n}\Big|^2}\geq e^{\lambda_n t_n^2}\right]
\leq e^{-\lambda_n t_n^2}\E \Bigg[e^{\lambda_n\left|\frac{\Vert X_n\Vert_2^2}{n}-\frac{\E \Vert X_n\Vert_2^2}{n}\right|^2}\Bigg]\cr
&\leq& e^{-\lambda_n t_n^2}\E \Bigg[e^{\lambda_n\left|\frac{\Vert X_n\Vert_2^2}{n}-\frac{\Vert \overline{X}_n\Vert_2^2}{n}\right|^2}\Bigg]
\leq  2|\alpha_*|\sqrt{2\pi n\varphi_M^{\prime\prime}(\alpha_*)}e^{-\lambda_n t_n^2}\E\Big[e^{\lambda_n\left|\frac{1}{n}\sum_{i=1}^n Y_i^{(2)}\right|^2}\Big]\big(1+o(1)\big).
\end{eqnarray*}
Hence, if $\lambda_n:=\frac{n}{16A^2}$ we obtain from \eqref{eq:expectation exp lambda n bounded by 2} that, for any sequence $(t_n)_{n=1}^\infty$,
$$
\Pro\left[\Bigg|\frac{\Vert X_n\Vert_2^2}{\E[\Vert X_n\Vert_2^2]}-1\Bigg|\geq \frac{nt_n}{\E[\Vert X_n\Vert_2^2]}\right]=\Pro\left[\Bigg|\frac{\Vert X_n\Vert_2^2}{n}-\frac{\E[\Vert X_n\Vert_2^2]}{n}\Bigg|\geq t_n\right]\leq  4|\alpha_*|\sqrt{2\pi n\varphi_M^{\prime\prime}(\alpha_*)}e^{-\frac{t_n^2n}{16A^2}}\big(1+o(1)\big),
$$
as $n\to\infty$. Since, like in Remark \ref{rmk:RemoveSquare}, for any $t_n\in(0,\infty)$
$$
\Pro\left[\Bigg|\frac{\Vert X_n\Vert_2}{\left(\E\Vert X_n\Vert_2^2\right)^{1/2}}-1\Bigg|\geq t_n\right]\leq\Pro\left[\Bigg|\frac{\Vert X_n\Vert_2^2}{\E\Vert X_n\Vert_2^2}-1\Bigg|\geq t_n\right],
$$
we obtain that, for any sequence $(t_n)_{n=1}^\infty\in(0,\infty)^\N$,
$$
\Pro\left[\Bigg|\frac{\Vert X_n\Vert_2}{\left(\E\Vert X_n\Vert_2^2\right)^{1/2}}-1\Bigg|\geq t_n\right]\leq 4|\alpha_*|\sqrt{2\pi n\varphi_M^{\prime\prime}(\alpha_*)}e^{-\frac{t_n^2n\left(\frac{\E\Vert X_n\Vert_2^2}{n}\right)^2}{16A^2}}\big(1+o(1)\big).
$$
All that is left to prove is that, as $n\to\infty$, $\frac{\E[\Vert X_n\Vert_2^2]}{n}=L_Z^2(1+o(1))$, which will be proved in Section \ref{sec:IsotropicConstant}.
\end{proof}

\section{Lower bounds on the probability of thin-shell width concentration}

In this section we establish lower bounds on the upper tail concentration probability. In particular, we are able to prove that our probabilistic estimates from Theorem \ref{thm:concentration thin-shell orlicz} are essentially sharp whenever $\frac{1}{n^{1/4}}\ll t_n\ll1$. We start with a first lower bound which reduces the problem to finding a good lower bound on the probability that simultaneously the normalized partial sums of the $Y_i^{(1)}$ and $Y_i^{(2)}$ lie in certain intervals.

\begin{lemma}\label{lem:LowerBoundAnyt}
Let $M:\R\to[0,\infty)$ be an Orlicz function, $R\in(0,\infty)$, $n\in\N$, and $\varphi_M:(-\infty,0)\to\R$ be the log-partition function with potential $M$, and let $(Z_i)_{i\in\N}$ be a sequence of independent random variables with Gibbs density $p_M$, i.e.,
$$
\varphi_M(\alpha)=\log \int_\R e^{\alpha M(x)}dx \quad\textrm{and}\quad p_M(x)=\frac{e^{\alpha_* M(x)}}{ \int_\R e^{\alpha_* M(x)}dx}.
$$
where $\alpha_*\in(-\infty,0)$ is chosen such that $\varphi_M^\prime(\alpha_*)=R$. Let for any $i\in\N$ $Y_i^{(1)}:=M(Z_i)-R$ and $Y_i^{(2)}:=Z_i^2-L_Z^2$,
where
\[
L_Z^2:=\frac{\int_\R x^2e^{\alpha_* M(x)}dx}{\int_\R e^{\alpha_* M(x)}dx}.
\]
If $X_n$ is a random vector uniformly distributed on $B_M^n(nR)$, then, for any $t>0$ and any sequence $(r_n)_{n\in\N}\in(0,\infty)^\N$, as $n\to\infty$
$$
\Pro\left[\Bigg|\frac{\Vert X_n\Vert_2^2}{n}-L_Z^2\Bigg|\geq t\right]
\geq
|\alpha_*|\sqrt{2\pi n\varphi^{\prime\prime}(\alpha_*)}
e^{\alpha_*r_n}\,\Pro\left[-1\leq\frac{1}{r_n}\sum_{i=1}^nY_i^{(1)}\leq0,\Bigg|\frac{1}{n}\sum_{i=1}^n Y_i^{(2)}\Bigg|\geq t\right]\big(1+o(1)\big),
$$
where the sequence $o(1)$ does not depend on $t$ or $(r_n)_{n\in\N}$.
\end{lemma}
\begin{proof}
As in the proof of Lemma \ref{lem:WidthAnyt}, if $X_n$ is a vector uniformly distributed on $B_M^n(nR)$, then, for any $t\in(0,\infty)$,
\begin{eqnarray*}
\Pro\left[\Bigg|\frac{\Vert X_n\Vert_2^2}{n}-L_Z^2\Bigg| \geq t\right]
&=&\frac{\E\left[\chi_{(-\infty,0]}\left(\sum_{i=1}^nY_i^{(1)}\right)\chi_{\R\setminus(-nt,nt)}\left(\sum_{i=1}^n Y_i^{(2)}\right)e^{-\alpha_*\sum_{i=1}^nY_i^{(1)}}\right]}{\E\left[\chi_{(-\infty,0]}\left(\sum_{i=1}^nY_i^{(1)}\right)e^{-\alpha_*\sum_{i=1}^nY_i^{(1)}}\right]}.
\end{eqnarray*}

Now, for any $r_n\in(0,\infty)$, we have that
\begin{eqnarray*}
&&\E\left[\chi_{(-\infty,0]}\left(\sum_{i=1}^nY_i^{(1)}\right)\chi_{\R\setminus(-nt,nt)}\left(\sum_{i=1}^n Y_i^{(2)}\right)e^{-\alpha_*\sum_{i=1}^nY_i^{(1)}}\right]\cr
&\geq&\E\left[\chi_{(-r_n,0]}\left(\sum_{i=1}^nY_i^{(1)}\right)\chi_{\R\setminus(-nt,nt)}\left(\sum_{i=1}^n Y_i^{(2)}\right)e^{-\alpha_*\sum_{i=1}^nY_i^{(1)}}\right]\cr
&\geq&e^{\alpha_*r_n}\E\left[\chi_{(-r_n,0]}\left(\sum_{i=1}^nY_i^{(1)}\right)\chi_{\R\setminus(-nt,nt)}\left(\sum_{i=1}^n Y_i^{(2)}\right)\right]\cr
&=&e^{\alpha_*r_n}\Pro\left[-1\leq\frac{1}{r_n}\sum_{i=1}^nY_i^{(1)}\leq0,\left|\frac{1}{n}\sum_{i=1}^n Y_i^{(2)}\right|\geq t\right].
\end{eqnarray*}
Again, the proof of \cite[Proposition 3.2]{KP2020} (see Equation \eqref{eq:asymptotic expectation} above) shows that
\[
\E\left[\chi_{(-\infty,0]}\left(\sum_{i=1}^nY_i^{(1)}\right)e^{-\alpha_*\sum_{i=1}^nY_i^{(1)}}\right]=\frac{1+o(1)}{|\alpha_*|\sqrt{2\pi n\varphi^{\prime\prime}(\alpha_*)}},
\]
and so we have that for any sequence $(r_n)_{n\in\N}$
$$
\Pro\left[\Bigg|\frac{\Vert X_n\Vert_2^2}{n}-L_Z^2\Bigg|\geq t\right]\geq|\alpha_*|\sqrt{2\pi n\varphi^{\prime\prime}(\alpha_*)}e^{\alpha_*r_n}\,\Pro\left[-1\leq\frac{1}{r_n}\sum_{i=1}^nY_i^{(1)}\leq0,\Bigg|\frac{1}{n}\sum_{i=1}^n Y_i^{(2)}\Bigg|\geq t\right]\big(1+o(1)\big),
$$
which completes the proof.
\end{proof}

Let us now present the proof of Theorem \ref{thm:lower bound}. By establishing an exponential equivalence, we show that under certain conditions the condition in the event in the previous lemma which comes from the $Y_i^{(1)}$ random variables is the dominating one. Combining this with an application of Petrov's moderate deviations result \cite{P1975} (see Theorem \ref{lem:MDPVector} above) to this first component, we obtain a lower bound. As explained before, the result shows that in the range $\frac{1}{n^{1/4}}\ll t_n\ll 1$ the asymptotic upper bound on the probability of concentration from Theorem \ref{thm:concentration thin-shell orlicz} is actually sharp.

\begin{proof}[Proof of Theorem \ref{thm:lower bound}]
First of all notice that for every two sequences $(r_n)_{n\in\N},(t_n)_{n\in\N}\in(0,\infty)^\N$  as in the statement we have that, calling $s_n:=t_n\sqrt{n}$ and $v_n:=\frac{r_n}{\sqrt{n}}$, for any $\varepsilon\in(0,1)$,
\begin{eqnarray*}
\Pro\left[-1\leq\frac{1}{r_n}\sum_{i=1}^nY_i^{(1)}\leq0,\Bigg|\frac{1}{n}\sum_{i=1}^n Y_i^{(2)}\Bigg|\geq t_n\right]&\geq&\Pro\left[-1\leq\frac{1}{v_n\sqrt{n}}\sum_{i=1}^nY_i^{(1)}\leq-\varepsilon,\Bigg|\frac{1}{s_n\sqrt{n}}\sum_{i=1}^n Y_i^{(2)}\Bigg|\geq 1\right]\cr
&=&\Pro\left[\left(\frac{1}{v_n\sqrt{n}}\sum_{i=1}^nY_i^{(1)},\frac{1}{s_n\sqrt{n}}\sum_{i=1}^nY_i^{(2)}\right)\in A_\varepsilon\right],
\end{eqnarray*}
where
$$
A_\varepsilon=\big\{(x,y)\in\R^2\,:\,x\in[-1,-\varepsilon],\,|y|\geq 1\big\}.
$$

Now we observe that
$$
\E\big[e^{uY_1^{(1)}}\big]=\frac{\int_{\R}e^{u(M(x)-R)}e^{\alpha_*M(x)}dx}{\int_{\R} e^{\alpha_*M(x)}dx}<\infty
$$
for every $u$ in a neighborhood of $0$. Therefore, by Theorem \ref{lem:MDPVector}, the sequence of random variables $\left(\frac{1}{v_n\sqrt{n}}\sum_{i=1}^nY_i^{(1)}\right)_{n\in\N}$ satisfies an LDP on $\R$ with speed $v_n^2$ and rate function $J:\R\to[0,\infty)$, $J(x):=\frac{x^2}{2\Var(M(Z)^2)}$. In particular, the sequence of random vectors $\left(\frac{1}{v_n\sqrt{n}}\sum_{i=1}^nY_i^{(1)},0\right)_{n\in\N}$ satisfies an LDP on $\R^2$ with speed $v_n^2$ and rate function $I:\R^2\to[0,\infty]$ given by
$$
I(x,y):=\begin{cases} \frac{x^2}{2\Var(M(Z)^2)}&:\, y=0\cr\infty &\textrm{ otherwise.}\end{cases}
$$

On the one hand, since $M(x)=\Omega(x^2)$ as $x\to\infty$, we have that
$$
\E\big[e^{uY_1^{(2)}}\big]=\frac{\int_{\R}e^{u(x^2-L_z^2)}e^{\alpha_*M(x)}dx}{\int_{\R} e^{\alpha_*M(x)}dx}<\infty
$$
for every $u$ in a neighborhood of $0$. Therefore, since $ 1\ll s_n=t_n\sqrt{n}\ll\sqrt{n}$, we obtain by means of Theorem \ref{lem:MDPVector} that, for every $\delta\in(0,\infty)$,
$$
\Pro\left[\Bigg|\frac{1}{s_n\sqrt{n}}\sum_{i=1}^nY_i^{(2)}\Bigg|>\delta\right]\leq e^{-\left(\frac{\delta^2}{2\Var(Z^2)}+o(1)\right)s_n^2}=e^{-\left(\frac{\delta^2}{2\Var(Z^2)}+o(1)\right)t_n^2n}.
$$
Thus, for every $\delta\in(0,\infty)$,
\begin{eqnarray*}
\limsup_{n\to\infty}\frac{\log\Pro\left[\left|\frac{1}{s_n\sqrt{n}}\sum_{i=1}^nY_i^{(2)}\right|>\delta\right]}{v_n^2}&\leq&\limsup_{n\to\infty}\frac{-\left(\frac{\delta^2}{2\Var(Z^2)}+o(1)\right)t_n^2n}{v_n^2}\cr
&=&-\left(\frac{\delta^2}{2\Var(Z^2)}+o(1)\right)\limsup_{n\to\infty}\frac{t_n^2n^2}{r_n^2}=-\infty,
\end{eqnarray*}
since $r_n/n\ll t_n$.
In view of Lemma \ref{lem:exponentially equivalent}, we see that the sequence $\left(\frac{1}{v_n\sqrt{n}}\sum_{i=1}^nY_i^{(1)},\frac{1}{s_n\sqrt{n}}\sum_{i=1}^nY_i^{(2)}\right)_{n\in\N}$ is exponentially equivalent to the sequence of random vectors $\left(\frac{1}{v_n\sqrt{n}}\sum_{i=1}^nY_i^{(1)},0\right)_{n\in\N}$ with speed $v_n^2$, and so
\[
\left(\frac{1}{v_n\sqrt{n}}\sum_{i=1}^nY_i^{(1)},\frac{1}{s_n\sqrt{n}}\sum_{i=1}^nY_i^{(2)}\right)_{n\in\N}
\]
satisfies an LDP (which on this scale is an MDP) with the same speed and the same rate function and, for every $\varepsilon\in(0,1)$,
\[
\Pro\left[\left(\frac{1}{v_n\sqrt{n}}\sum_{i=1}^nY_i^{(1)},\frac{1}{s_n\sqrt{n}}\sum_{i=1}^nY_i^{(2)}\right)\in A_\varepsilon\right]\geq e^{-\left(\frac{\varepsilon^2}{2\Var(M(Z)^2)}+o(1)\right)v_n^2}=e^{-\left(\frac{\varepsilon^2}{2\Var(M(Z)^2)}+o(1)\right)\frac{r_n}{n}r_n}.
\]
By Lemma \ref{lem:LowerBoundAnyt}, since $\frac{r_n}{n}\ll1$, we have that
\begin{eqnarray*}
\Pro\left[\Bigg|\frac{\Vert X_n\Vert_2^2}{n}-L_Z^2\Bigg|\geq t_n\right]&\geq&|\alpha_*|\sqrt{2\pi n\varphi^{\prime\prime}(\alpha_*)} e^{\alpha_*r_n}e^{-\left(\frac{\varepsilon^2}{2\Var(M(Z)^2)}+o(1)\right)\frac{r_n}{n}r_n}\big(1+o(1)\big)\cr
&=&|\alpha_*|\sqrt{2\pi n\varphi^{\prime\prime}(\alpha_*)} e^{\alpha_*r_n\big(1+o(1)\big)}\big(1+o(1)\big)\cr
&=&|\alpha_*|\sqrt{2\pi n\varphi^{\prime\prime}(\alpha_*)}e^{-r_n\big(-\alpha_*+o(1)\big)}\big(1+o(1)\big).
\end{eqnarray*}
The second part of the statement follows by setting $r_n:=t_n^2n$, which means that $t_n=\sqrt{\frac{r_n}{n}}$ and so the result holds for any $\frac{1}{n^{1/4}}\ll t_n\ll1$.
\end{proof}

\begin{rmk}\label{rem:discussion of conditions}
Let us say a few words about the conditions $\frac{1}{\sqrt{n}}\ll \frac{r_n}{n}\ll t_n \ll 1$ in Theorem \ref{thm:lower bound}. An inspection of the proof shows that the conditions $\frac{1}{\sqrt{n}}\ll \frac{r_n}{n} \ll 1$ are needed in order to ensure that  $\frac{1}{v_n\sqrt{n}}\sum_{i=1}^n Y_i^{(1)}$, where $v_n:=\frac{r_n}{\sqrt{n}}$, is on the scale of a moderate deviations principle, which then allows us to apply Petrov's result (Theorem \ref{lem:MDPVector}). The condition $\frac{r_n}{n}\ll t_n$ is related to Lemma \ref{lem:LowerBoundAnyt} in the following way: if we take $t_n\ll \frac{r_n}{n}$ and assume that $M(x)=x^2$, which means that $Y_i^{(1)}$ is equal to $Y_i^{(2)}$, then Lemma 6.1 gives $0$ as a trivial lower bound. This can be seen by looking at the term
\[
\E\left[\chi_{(-r_n,0]}\left(\sum_{i=1}^nY_i^{(2)}\right)\chi_{\R\setminus(-nt,nt)}\left(\sum_{i=1}^n Y_i^{(2)}\right)e^{-\alpha_*\sum_{i=1}^nY_i^{(2)}}\right],
\]
appearing in the proof of the lower bound, because then we integrate the exponential over the set $(-r_n,0]\cap \R\setminus(-nt_n,nt_n)=\emptyset$ as $n\to\infty$.
\end{rmk}

Using the Lemma \ref{lem:LowerBoundAnyt} in combination with the vector-version of the moderate deviation principle due to Petrov \cite{P1975} (see Theorem \ref{lem:MDPVector}), we obtain the following estimate which complements the one from Theorem \ref{thm:lower bound} in the sense that there is no wiggle room between $\frac{r_n}{n}$ and $t_n$, but now $\frac{r_n}{n}=t_n$.

\begin{thm}\label{lem:LowerEstimate}
Let $M:\R\to[0,\infty)$ be an Orlicz function such that $M\in\Omega(x^2)$ as $x\to\infty$, $n\in\N$, $R\in(0,\infty)$, and $\varphi_M:(-\infty,0)\to\R$ be the log-partition function with potential $M$, and let $Z$ be a random variable with Gibbs density $p_M$, i.e.,
$$
\varphi_M(\alpha)=\log \int_\R e^{\alpha M(x)}dx\quad\textrm{and}\quad p_M(x)=\frac{e^{\alpha_* M(x)}}{ \int_\R e^{\alpha_* M(x)}dx}.
$$
where $\alpha_*\in(-\infty,0)$ is chosen such that $\varphi_M^\prime(\alpha_*)=R$.
Assume that the covariance matrix of $(M(Z)-R, Z^2-L_Z^2)$ is invertible and that $X_n$ is a random vector uniformly distributed on $B_M^n(nR)$. Then, for every sequence $(t_n)_{n\in\N}\in(0,\infty)^\N$ such that $\frac{1}{\sqrt{n}}\ll t_n\ll1$, we obtain, as $n\to\infty$,
$$
\Pro\left[\Bigg|\frac{\Vert X_n\Vert_2^2}{n}-L_Z^2\Bigg|\geq t_n\right]\geq e^{-t_nn(-\alpha_*+o(1))},
$$
where
$$
L_Z^2:=\frac{\int_\R x^2e^{\alpha_* M(x)}dx}{\int_\R e^{\alpha_* M(x)}dx}.
$$
\end{thm}
\begin{proof}
Let us call $s_n:=t_n\sqrt{n}$ and $r_n:=s_n\sqrt{n}=t_n n$. Then $1\ll s_n\ll\sqrt{n}$ and
\begin{eqnarray*}
\Pro\left[-1\leq\frac{1}{r_n}\sum_{i=1}^nY_i^{(1)}\leq0,\Bigg|\frac{1}{n}\sum_{i=1}^n Y_i^{(2)}\Bigg|\geq t_n\right]&=&\Pro\left[-1\leq\frac{1}{s_n\sqrt{n}}\sum_{i=1}^nY_i^{(1)}\leq0,\Bigg|\frac{1}{s_n\sqrt{n}}\sum_{i=1}^n Y_i^{(2)}\Bigg|\geq 1\right]\cr
&=&\Pro\left[\frac{1}{s_n\sqrt{n}}\sum_{i=1}^n(Y_i^{(1)},Y_i^{(2)})\in A\right],
\end{eqnarray*}
where
$$
A=\big\{(x,y)\in\R^2\,:\,x\in[-1,0],\,|y|\geq 1\big\}.
$$
Since $M(x)=\Omega(x^2)$ as $x\to\infty$ we have that
$$
\E\big[e^{\big\langle u,(Y_1^{(1)}, Y_1^{(2)})\big\rangle}\big]=\frac{\int_{\R}e^{u_1(M(x)-R)+u_2(x^2-L_z^2)}e^{\alpha_*M(x)}dx}{\int_{\R} e^{\alpha_*M(x)}dx}<\infty
$$
for all $u=(u_1,u_2)$ in a neighborhood of $0$. Thus, as the covariance matrix $\bC$ of $(Y_1^{(1)}, Y_1^{(2)})=(M(Z)-R, Z^2-L_Z^2)$ is invertible, we have by Theorem \ref{lem:MDPVector} that
$$
\Pro\left[\frac{1}{s_n\sqrt{n}}\sum_{i=1}^n(Y_i^{(1)},Y_i^{(2)})\in A\right]=e^{-s_n^2\big(\inf_{(x,y)\in A}I(x,y)+o(1)\big)}=e^{-t_n^2n\big(\inf_{(x,y)\in A}I(x,y)+o(1)\big)},
$$
where
$$
I(x,y)=\langle (x,y), C^{-1}(x,y)\rangle.
$$
Notice that necessarily
$$
\inf_{(x,y)\in A}I(x,y)=\min_{{x\in[0,1]}\atop{y=\pm1}} I(x,y).
$$
By Lemma \ref{lem:LowerBoundAnyt}
\begin{eqnarray*}
\Pro\left[\Bigg|\frac{\Vert X_n\Vert_2^2}{n}-L_Z^2\Bigg|\geq t_n\right]&\geq& e^{\alpha_*r_n}e^{-t_n^2n\big(\min_{{x\in[0,1]}\atop{y=\pm1}}I(x,y)+o(1)\big)}=e^{-t_nn\big((\min_{{x\in[0,1]}\atop{y=\pm1}}I(x,y)+o(1))t_n-\alpha_*\big)}\cr
&=&e^{-t_nn(-\alpha_*+o(1))},
\end{eqnarray*}
which completes the proof.
\end{proof}

\section{The asymptotic value of the isotropic constant of Orlicz balls}\label{sec:IsotropicConstant}

We shall now present the proof of Theorem \ref{thm:isotropic constant orlicz} in which we determine the precise asymptotic value of the isotropic constant of Orlicz balls, i.e., if $M:\R\to[0,\infty)$ is an Orlicz function and $R\in(0,\infty)$, then we show that
$$
\lim_{n\to\infty}L_{B_M^n(nR)}=\frac{e^{\alpha_*R}}{\int_\R e^{\alpha_*M(x)}dx}\left(\frac{\int_\R x^2e^{\alpha_* M(x)}dx}{\int_\R e^{\alpha_* M(x)}dx}\right)^{1/2},
$$
where $\alpha_*\in(-\infty,0)$ is the unique element such that the log-partition function
$$
\varphi_M:(-\infty,0)\to\R,\qquad \varphi_M(\alpha)=\log \int_\R e^{\alpha M(x)}dx.
$$
verifies that $\varphi_M^\prime(\alpha_*)=R$.

The following lemma is essential for our proof. Once we have obtained the result the theorem will be an easy consequence by combining the lemma with the recently obtained formula for the asymptotic volume of Orlicz balls in \cite{KP2020}.

\begin{lemma}
Let $M:\R\to[0,\infty)$ be an Orlicz function, $n\in\N$, and $R\in(0,\infty)$. If $X_n$ is a random vector uniformly distributed on $B_M^n(nR)$, then
$$
\lim_{n\to\infty}\E\Bigg[\frac{\Vert X_n\Vert_2^2}{n}\Bigg]=L_Z^2,
$$
where
$$
L_Z^2:=\frac{\int_\R x^2e^{\alpha_* M(x)}dx}{\int_\R e^{\alpha_* M(x)}dx}.
$$
\end{lemma}

\begin{proof}
We have that
\begin{eqnarray*}
\E\Bigg[\frac{\Vert X_n\Vert_2^2}{n}\Bigg]&=&\int_0^\infty 2t\,\Pro\left[\frac{\Vert X\Vert_2}{\sqrt{n}}\geq t\right]dt\cr
&=&\int_0^{L_Z}2t\,\Pro\left[\frac{\Vert X_n\Vert_2}{\sqrt{n}}\geq t\right]dt+\int_{L_Z}^\infty2t\,\Pro\left[\frac{\Vert X_n\Vert_2}{\sqrt{n}}\geq t\right]dt.
\end{eqnarray*}
Since for every $t\in[0, L_Z)$ we have, by Lemma \ref{lem: thin-shell concetration via chebyshev}, that
\begin{itemize}
\item $\displaystyle{\lim_{n\to\infty}\Pro\left[\frac{\Vert X_n\Vert_2}{\sqrt{n}}\geq t\right]\to1}$
\item $\Pro\left[\frac{\Vert X_n\Vert_2}{\sqrt{n}}\geq t\right]\leq 1$, which is integrable on $[0,L_z]$,
\end{itemize}
we obtain by the dominated convergence theorem that
$$
\lim_{n\to\infty}\int_0^{L_Z}2t\,\Pro\left[\frac{\Vert X_n\Vert_2}{\sqrt{n}}\geq t\right]dt=\int_0^{L_Z}2tdt=L_Z^2.
$$
On the other hand, since $\widetilde{B}_M^n(nR)=\frac{B_M^n(nR)}{\vol_n(B_M^n(nR))^{1/n}}$ is in isotropic position and its isotropic constant is bounded by an absolute constant (as it is a 1-symmetric convex body) and, since by \cite[Theorem A]{KP2020} (see Equation \eqref{eq:log-volume orlicz} above)
$$
\lim_{n\to\infty}\vol_n\big(B_M^n(nR)\big)^{1/n}=e^{\varphi(\alpha_*)-\alpha_* R}=\frac{\int_\R e^{\alpha_*M(x)}dx}{e^{\alpha_*R}},
$$
we have that there exists a constant $C(M,R)\in(0,\infty)$ and $n_0\in\N$ such that if $n\geq n_0$, then
$$
\E\big[\Vert X_n\Vert_2^2\big]=nL_{B_M^2(nR)}^2\vol_n\big(B_M^n(nR)\big)^{2/n}\leq C(M,R)n.
$$
As a consequence of Borell's inequality (see, e.g., \cite[Corollary 3.2.17]{IsotropicConvexBodies}), there exists an absolute constant $C\in(0,\infty)$ such that, for every $t\geq 0$,
$$
\Pro\left[\Vert X_n\Vert_2\geq C\, \E \left[\Vert X_n\Vert_2^2\right]^{1/2}t\right]\leq 2e^{-t^2}.
$$
Therefore, there exists some constant $c(M,R)\in(0,\infty)$ and $n_0\in\N$ such that if $n\geq n_0$ and $t \geq0$, then
$$
\Pro\left[\frac{\Vert X_n\Vert_2}{\sqrt{n}}\geq  t\right]\leq 2e^{-ct^2}.
$$
Thus, if $t> L_Z$, then
\begin{itemize}
\item $\lim_{n\to\infty}2t\,\Pro\left[\frac{\Vert X_n\Vert_2}{\sqrt{n}}\geq t\right]=0$
\item $2t\,\Pro\left[\frac{\Vert X_n\Vert_2}{\sqrt{n}}\geq t\right]\leq 4te^{-ct^2}$, which is integrable on $[L_Z,\infty)$
\end{itemize}
and, by the dominated convergence theorem,
$$
\lim_{n\to\infty}\int_{L_Z}^{\infty}2t\,\Pro\left[\frac{\Vert X_n\Vert_2}{\sqrt{n}}\geq t\right]dt=\int_{L_Z}^\infty0dt=0,
$$
which completes the proof.
\end{proof}

The proof of Theorem \ref{thm:isotropic constant orlicz} is now an easy consequence of the previous lemma and the asymptotic formula for the volume of Orlicz balls from \cite[Theorem A]{KP2020}.

\begin{proof}[Proof of Theorem \ref{thm:isotropic constant orlicz}]
Since $\widetilde{B}_M^n(nR)=\frac{B_M^n(nR)}{\vol_n(B_M^n(nR))^{1/n}}$ is in isotropic position, we have that
$$
L_{B_M^n(nR)}^2=\frac{1}{\vol_n(B_M^n(nR))^{2/n}}\E\Bigg[\frac{\Vert X_n\Vert_2^2}{n}\Bigg],
$$
where $X_n$ is a random vector uniformly distributed on $B_m^n(nR)$. Since by \cite[Theorem A]{KP2020} (see Equation \eqref{eq:log-volume orlicz} above), we have
$$
\lim_{n\to\infty}\vol_n\big(B_M^n(nR)\big)^{1/n}=e^{\varphi(\alpha_*)-\alpha_* R}=\frac{\int_\R e^{\alpha_*M(x)}dx}{e^{\alpha_*R}},
$$
we obtain the theorem.
\end{proof}

\section{The asymptotic thin-shell width concentration for $\ell_p^n$ balls}

In this section we are going to consider the particular case of the $\ell_p^n$ balls, in which $M(t)=|t|^p$, with $p\geq1$. In that case, for any $R\in(0,\infty)$,
$$
B_M^n(nR)=\Big\{x\in\R^n\,:\,\sum_{i=1}^n|x_i|^p\leq nR\Big\}=(nR)^{1/p}B_p^n
$$
and $\varphi_p:(-\infty,0)\to\R$ is defined as
$$
\varphi_p(\alpha)=\log \int_\R e^{\alpha|t|^p}dt=\log 2+\log\Gamma\left(1+\frac{1}{p}\right)-\frac{1}{p}\log(-\alpha).
$$
This means that
\[
\varphi_p^\prime(\alpha)=-\frac{1}{p\alpha},\,\,\, \quad\alpha_*=-\frac{1}{pR},\,\,\, \varphi_p^{\prime\prime}(\alpha)=\frac{1}{p\alpha^2},\,\,\,\text{and}\,\,\,\varphi_p^{\prime\prime}(\alpha_*)=pR^2.
\]
and the Gibbs density is given by
$$
p(x)=\frac{e^{-\frac{|t|^p}{pR}}dt}{\int_\R e^{-\frac{|t|^p}{pR}}dt}=\frac{e^{-\frac{|t|^p}{pR}}dt}{2(pR)^{1/p}\Gamma\left(1+\frac{1}{p}\right)}.
$$
Since for any $R\in(0,\infty)$ the Orlicz balls are merely dilations of the same convex body, from now on we are going to assume that $R=1$. In such case the random variable $Z$ with density $p(x)$ is called a $p$-generalized Gaussian and
\[
L_Z^2=\E [Z^2]=\frac{p^{2/p}\Gamma\left(1+\frac{3}{p}\right)}{3\Gamma\left(1+\frac{1}{p}\right)}\quad\text{and}\quad \Var[Z^2]=\frac{p^{4/p}\left(9\Gamma\left(1+\frac{5}{p}\right)\Gamma\left(1+\frac{1}{p}\right)-5\Gamma\left(1+\frac{3}{p}\right)^2\right)}{45\Gamma\left(1+\frac{1}{p}\right)^2}.
\]
Then if $X_n$ is a random vector uniformly distributed on $n^{1/p}B_p^n$, we have, for any $t\in(0,\infty)$,
$$
\Pro\left[\Bigg|\frac{\Vert X_n\Vert_2^2}{n}-L_Z^2\Bigg|\geq t\right]
\leq
\sqrt{\frac{2\pi n}{p}}\Pro\left[\Bigg|\frac{1}{n}\sum_{i=1}^n (Z_i^2-L_Z^2)\Bigg|\geq t\right]\big(1+o(1)\big).
$$

In the case of the $\ell_p^n$ balls, Theorem \ref{thm:concentration thin-shell orlicz} gives that if $X_n$ is uniformly distributed on $n^{1/p}B_p^n$ and  $p\geq 2$, then if $\frac{1}{\sqrt{n}}\ll t_n\ll1$, as $n\to\infty$,
$$
\Pro\left[\Bigg|\frac{\Vert X_n\Vert_2^2}{nL_Z^2}-1\Bigg|\geq t_n\right]\leq\sqrt{\frac{2\pi n}{p}}e^{-\frac{t_n^2nL_Z^4(1+o(1))}{2\Var(Z^2)}}\big(1+o(1)\big)
$$
and, together with Theorem \ref{thm:lower bound}, if $p\geq 2$ and $\frac{1}{n^{1/4}}\ll t_n\ll1$, then
$$
\sqrt{\frac{2\pi n}{p}}e^{-\frac{t_n^2nL_Z^4}{p}\big(1+o(1)\big)}(1+o(1))\leq\Pro\left[\Bigg|\frac{\Vert X_n\Vert_2^2}{nL_Z^2}-1\Bigg|\geq t_n\right]\leq\sqrt{\frac{2\pi n}{p}}e^{-\frac{t_n^2nL_Z^4(1+o(1))}{2\Var(Z^2)}}\big(1+o(1)\big).
$$

Theorem \ref{thm:thin-shell concetration lp balls 1<p<2}, which refers to the situation when $1\leq p<2$ will be a direct consequence of the following lemma.

\begin{lemma}\label{lem:MDP Z^2}
Let $1\leq p<2$ and $(Z_i)_{i\in\N}$ be a sequence of independent copies of a $p$-generalized Gaussian random variable with density
$$
p(x)=\frac{e^{-\frac{|x|^p}{p}}dx}{2p^\frac{1}{p}\Gamma\left(1+\frac{1}{p}\right)}.
$$
and let $$L_Z^2=\E [Z^2]=\frac{p^{\frac{2}{p}}\Gamma\left(1+\frac{3}{p}\right)}{3\Gamma\left(1+\frac{1}{p}\right)}.$$ Let $(s_n)_{n\in\N}\in(0,\infty)^\N$ be a sequence such that $1\ll s_n\ll n^{\frac{1}{2(4/p-1)}}$. Then the sequence of random variables
\[
\left(\frac{1}{s_n\sqrt{n}}\sum_{i=1}^n\big(Z_i^2-L_Z^2\big)\right)_{n\in\N}
\]
satisfies an LDP on $\R$ with speed $s_n^2$ and good rate function
\[
I:\R\to[0,\infty),\qquad I(x)= \frac{x^2}{2\Var[|Z_1|^2]}={x^2\over 2\Var [Z_1^2]}.
\]
\end{lemma}
\begin{proof}
Clearly, $\E[Z_1^2-L_Z^2]=0$ and $\Var[Z_1^2]>0$. Thus, all that remains to check is condition \eqref{eq:limit condition eichelsbacher-loewe}. We first observe that (see, e.g., \cite[Lemma 4.2]{GKR2017})
\[
\int_a^\infty e^{-y^p/p} \,\dint y \leq \frac{1}{a^{p-1}}e^{-a^p/p},\qquad a\in(0,\infty).
\]
Using this tail estimate, we obtain (for sufficiently large $n\in\N$)
\begin{align*}
& \frac{1}{s_n^2} \log\Big(n\,\Pro\big[\big|Z_1^2-L_Z^2\big| > \sqrt{n}s_n\big]\Big) \\
&\qquad\qquad \leq \frac{1}{s_n^2} \log\Big(n\,\Pro\big[Z_1^q > \sqrt{n}s_n + L_Z^2\big] + n\,\Pro\big[-Z_1^2 > \sqrt{n}s_n - L_Z^2\big]\Big) \\
&\qquad\qquad = \frac{1}{s_n^2} \log\Big(n\,\Pro\big[Z_1^2 > \sqrt{n}s_n + L_Z^2\big]\Big) \\
&\qquad\qquad \leq  \frac{1}{s_n^2} \log\Big(2n\,\Pro\Big[Z_1 > \big(\sqrt{n}s_n + L_Z^2\big)^{1/2}\Big]\Big) \\
&\qquad\qquad \leq \frac{1}{s_n^2} \log\bigg(\frac{n}{p^{1/p}\Gamma(1+1/p)}\cdot\frac{1}{\big(\sqrt{n}s_n + L_Z^2\big)^{(p-1)/q}}e^{-\frac{1}{p}\,\big(\sqrt{n}s_n + L_Z^2\big)^{p/2}}\bigg)\\
&\qquad\qquad \leq \frac{1}{s_n^2} \log\bigg(\frac{n}{p^{1/p}\Gamma(1+1/p)}e^{-\frac{1}{p}\,\big(\sqrt{n}s_n + L_Z^2\big)^{p/2}}\bigg)\\
&\qquad\qquad = \frac{1}{s_n^2} \log\bigg(e^{-\frac{1}{p}\,\big(\sqrt{n}s_n + L_Z^2\big)^{p/2} + \log\frac{n}{p^{1/p}\Gamma(1+1/p)}}\bigg)\\
&\qquad\qquad \leq \frac{1}{s_n^2} \log\bigg(e^{-c_p\,\big(\sqrt{n}s_n + L_Z^2\big)^{p/2}} \bigg),
\end{align*}
where $c_p\in(0,\infty)$ is a suitable constant depending only on $p$. Therefore, we get
\[
 \frac{1}{s_n^2} \log\Big(n\,\Pro\big[\big|Z_1^2-L_Z^2\big| > \sqrt{n}s_n\big]\Big) \leq - \frac{c_p\,\big(\sqrt{n}s_n + L_Z^2\big)^{p/2}}{s_n^2} \leq - \frac{c_p\,\big(\sqrt{n}s_n\big)^{p/2}}{s_n^2} = - c_p\frac{n^{p/4}}{s_n^{2-p/2}} \to - \infty,
\]
as $n\to\infty$, where in the last step we used that $s_n=o\big(n^{\frac{1}{2(4/p-1)}}\big)$. This shows that
\[
\lim_{n\to\infty} \frac{1}{s_n^2} \log\Big(n\,\Pro\Big[\big|Z_1^2-L_Z^2\big| > \sqrt{n}s_n\Big]\Big) = -\infty
\]
and so the result is a direct consequence of Theorem \ref{lem:eicheslbacher-loewe}.
\end{proof}

We can now prove Theorem \ref{thm:thin-shell concetration lp balls 1<p<2}.
\begin{proof}[Proof of Theorem \ref{thm:thin-shell concetration lp balls 1<p<2}]
Let $1\leq p<2$, and $\frac{1}{\sqrt{n}}\ll t_n\ll\frac{1}{n^{\frac{4-2p}{2(4-p)}}}$. Calling $s_n= t_n\sqrt{n}$ we have that $1\ll s_n\ll n^{\frac{1}{2(4/p-1)}}$ and so, by Lemma \ref{lem:MDP Z^2}, calling $ Y_i^{(2)}=Z_i^2-L_Z^2$ we have that, as $n\to\infty$,
\begin{eqnarray*}
\Pro\left[\Bigg|\frac{1}{n}\sum_{i=1}^n Y_i^{(2)}\Bigg|\geq t_n\right]
&=&\Pro\left[\Bigg|\frac{1}{s_n\sqrt{n}}\sum_{i=1}^n Y_i^{(2)}\Bigg|\geq 1\right]=e^{-I(1)(1+o(1))s_n^2}
=e^{-\frac{t_n^2n(1+o(1))}{2\Var(Z^2)}},
\end{eqnarray*}
where $I(x)={x^2\over 2\Var [Z_1^2]}$. Consequently, if $1\leq p<2$ and $\frac{1}{\sqrt{n}}\ll t_n\ll \frac{1}{n^{\frac{4-2p}{2(4-p)}}}$, then by Lemma \ref{lem:WidthAnyt}, as $n\to\infty$,
$$
\Pro\left[\Bigg|\frac{\Vert X_n\Vert_2^2}{n}-L_Z^2\Bigg|\geq t_n\right]\leq\sqrt{\frac{2\pi n}{p}}e^{-\frac{t_n^2n(1+o(1))}{2\Var(Z^2)}}\big(1+o(1)\big).
$$
Thus, as $n\to\infty$,
$$
\Pro\left[\Bigg|\frac{\Vert X_n\Vert_2^2}{nL_Z^2}-1\Bigg|\geq t_n\right]\leq\sqrt{\frac{2\pi n}{p}}e^{-\frac{t_n^2nL_Z^4(1+o(1))}{2\Var(Z^2)}}\big(1+o(1)\big).
$$
The proof of the lower bound is the same as in Theorem \ref{thm:lower bound}, where  now we have the restriction $\frac{1}{\sqrt{n}}\ll\frac{r_n}{n}\ll t_n\ll \frac{1}{n^{\frac{4-2p}{2(4-p)}}}$ as $n\to\infty$ and then, choosing $r_n=t_n^2 n$, forces the restriction $\frac{4}{3}<p<2$ and $\frac{1}{n^{1/4}}\ll t_n\ll\frac{n^{\frac{3p-4}{4(4-p)}}}{n^{1/4}}$. This completes the proof.
\end{proof}

\begin{rmk}
If $p\geq 2$, by means of Cram\'er's theorem (see Theorem \ref{thm:Cramer}) we can also obtain that for every fixed $t\in(0,\infty)$
$$
\Pro\left[\Bigg|\frac{1}{n}\sum_{i=1}^n (Z_i^2-L_Z^2)\Bigg|\geq t\right]=e^{-\frac{1}{p}f(t)n\big(1+o(1)\big)},
$$
where $f(t)=\inf_{|s|\geq t}\Lambda^*(s)$ with $\Lambda^*$ being the Legendre transform of the function
$$
\Lambda(u)=\log \E \big[e^{u(Z^2-L_Z^2)}\big]=\frac{\int_R e^{u(x^2-L_Z^2)}e^{-\frac{|x|^p}{p}}dx}{2p^\frac{1}{p}\Gamma\left(1+\frac{1}{p}\right)},
$$
which is finite on a neighborhood of $0$ if $p\geq 2$. So it $p\geq 2$, then we obtain, for every fixed $t\in(0,\infty)$,
$$
\Pro\left[\Bigg|\frac{\Vert X_n\Vert_2^2}{nL_Z^2}-1\Bigg|\geq t\right]\leq\sqrt{\frac{2\pi n}{p}}e^{-\frac{1}{p}f(tL_Z^2)n\big(1+o(1)\big)}\big(1+o(1)\big).
$$
If $1\leq p<2$, then  we have that
the sequence of random variables $\frac{1}{n}\sum_{i=1}^nZ_i^2$ satisfies an LDP with speed $n^{p/2}$ and rate function
$$
I(x)=\begin{cases}\frac{1}{p}(x-L_Z^2)^{\frac{p}{2}}&:\,x\geq L_Z^2\cr\infty&\textrm{otherwise},
\end{cases}
$$
a fact that was proved in \cite[Proof of Theorem 1.2]{APT2018}.
Hence, for every fixed $t\in(0,\infty)$,
$$
\Pro\left[\Bigg|\frac{1}{n}\sum_{i=1}^n (Z_i^2-L_Z^2)\Bigg|\geq t\right]=e^{-\frac{1}{p}t^\frac{p}{2}n^\frac{p}{2}\big(1+o(1)\big)}.
$$
Therefore, for every fixed $t\in(0,\infty)$,
$$
\Pro\left[\Bigg|\frac{\Vert X_n\Vert_2^2}{nL_Z^2}-1\Bigg|\geq t\right]\leq\sqrt{\frac{2\pi n}{p}}e^{-\frac{1}{p}t^\frac{p}{2}n^\frac{p}{2}L_Z^p\big(1+o(1)\big)}\big(1+o(1)\big).
$$
\end{rmk}

\subsection*{Acknowledgement}
JP is supported by the Austrian Science Fund (FWF) Project P32405 \textit{Asymptotic geometric analysis and applications} and Austrian Science Fund  (FWF)  Project  F5513-N26, which is a part of a Special
Research Program.
DA-G is supported by MICINN Project PID-105979-GB-I00 and DGA Project E48\_20R.

\bibliographystyle{plain}
\bibliography{thinshell}

\bigskip
\bigskip
	
	\medskip
	
	\small

	\noindent \textsc{David Alonso-Guti\'errez:} \'Area de an\'alisis matem\'atico, Departamento de matem\'aticas, Facultad de Ciencias, Universidad de Zaragoza, Pedro Cerbuna 12, 50009 Zaragoza (Spain), IUMA
		
	\noindent
		{\it E-mail:} \texttt{alonsod@unizar.es}
	
		\medskip
	
	\noindent \textsc{Joscha Prochno:} Institute of Mathematics and Scientific Computing,
	University of Graz, Heinrichstrasse 36, 8010 Graz, Austria
	
	\noindent
	{\it E-mail:} \texttt{joscha.prochno@uni-graz.at}

\end{document}